\documentclass[11pt]{amsart} 
\usepackage[utf8]{inputenc}
\usepackage{amsmath}
\usepackage{amssymb}
\usepackage{amsthm}
\usepackage[top=1in, bottom=1in, left=1in, right=1in]{geometry}
\usepackage{xcolor}
\usepackage{hyperref}
\usepackage{graphicx}
\usepackage{mathtools}
\usepackage[all]{xy}
\usepackage[dvipsnames]{xcolor}
\usepackage{tikz-cd}
\usepackage{quiver}
\usepackage{spectralsequences}
\usepackage{biblatex}
\usepackage{ytableau}
\usepackage{tabularx}
\usepackage{booktabs}
\usepackage[labelfont=bf,format=plain,justification=raggedright,singlelinecheck=false]{caption}

\bibliography{Gilbert_Sources} 

\newtheorem{theorem}{Theorem}[section]
\newtheorem{lemma}[theorem]{Lemma}
\newtheorem{prop}[theorem]{Proposition}
\newtheorem{remark}[theorem]{Remark}
\newtheorem{definition}{Definition}[section]

\newtheorem{conjecture}{Conjecture}

\newtheorem{coro}{Corollary}[section]
\newtheorem*{theorem*}{Theorem}
\newtheorem*{prop*}{Proposition}
\newtheorem*{lemma*}{Lemma}
\newtheorem*{coro*}{Corollary}
\newtheorem*{definition*}{Definition}
\newtheorem*{conjecture*}{Conjecture}
\newtheorem*{claim*}{Claim}
\newtheorem{example}{Example}

\newcommand{\aff}{\mathbb{A}}
\newcommand{\GG}{\mathbb{G}}
\newcommand{\N}{\mathbb{N}}
\newcommand{\Z}{\mathbb{Z}}
\newcommand{\Q}{\mathbb{Q}}
\newcommand{\R}{\mathbb{R}}
\newcommand{\C}{\mathbb{C}}
\newcommand{\p}{\mathbb{P}}

\newcommand{\ff}{\mathfrak{f}}
\newcommand{\RR}{\mathfrak{R}}

\newcommand{\XX}{\mathfrak{X}}
\newcommand{\zz}{\mathfrak{z}}

\newcommand{\Acal}{\mathcal{A}}
\newcommand{\BB}{\mathcal{B}}
\newcommand{\CC}{\mathcal{C}}
\newcommand{\DD}{\mathcal{D}}

\newcommand{\II}{\mathcal{I}}

\newcommand{\Ocal}{\mathcal{O}}

\newcommand{\eex}{\mathcal{X}}
\newcommand{\eey}{\mathcal{Y}}

\newcommand{\Ext}{\operatorname{Ext}}
\newcommand{\Hilb}{\mathcal{H}ilb}

\newcommand{\im}{\operatorname{im}}
\newcommand{\Ch}{\widetilde{\operatorname{Ch}}}
\newcommand{\Td}{\widetilde{\operatorname{Td}}}
\newcommand{\Rep}{\operatorname{Rep}}
\newcommand{\Irr}{\operatorname{Irr}}
\newcommand{\Span}{\operatorname{Span}}
\newcommand{\Hilbs}{\operatorname{Hilb}}
\newcommand{\ext}{\operatorname{ext}}
\newcommand{\Coh}{\operatorname{Coh}}

\newcommand{\Spec}{\operatorname{Spec}}
\newcommand{\modn}{\text{ (mod }n)}
\newcommand{\fgMod}{\operatorname{f.g.Mod}}
\newcommand{\Fix}{\operatorname{Fix}}

\newcommand{\innprod}[2]{\langle #1, #2 \rangle}
\newcommand{\Orb}{\operatorname{Orb}}

\newcommand{\rhostd}{\rho_{\operatorname{std}}}
\newcommand{\eexgr}{\eex_{\operatorname{rig}}}

\title{\large $G$-Equivariant Hilbert Schemes and Quadratic Forms}
\author{Samuel ``Jay" Bachhuber, Amin Gholampour}
\date{}

\begin{document}
\begin{abstract}
  Let $G$ be a finite abelian group acting linearly on $\aff^2$ 
  and $\eex = [\aff^2/G]$ denote the associated quotient stack.  
  Following the work of Olsson and Starr in \cite{Olsson_Starr_2008},
  for any $\rho \in \Rep(G)$ we will analyze a fine moduli space 
  $\eex^{[\rho]}$ which parameterizes zero dimensional 
  substacks $\zz \subset \eex$ such that if $\xi$ is the $G$-equivariant 
  subscheme of $\aff^2$ corresponding to $\zz$ then 
  $H^0(\Ocal_{\xi}) \simeq \rho$ as representations. When $G$ is an abelian 
  group, by means of a Grothendieck-Riemann-Roch theorem of To\"en 
  we give explicit representation theoretic formulas of $\dim(\eex^{[\rho]})$ 
  similar to the McKay correspondence. This is used to give a classification 
  of all such $\eex$ according to a quadratic form associated to 
  $\eex$ called the Euler form. For $[\aff^2/\mu_n]$ given by 
  the cyclic action $\zeta \cdot (x,y) = (\zeta x, \zeta^{-1} y)$, 
  we give an interpretation of the Euler form in terms of border strips 
  which may be removed by a diagram, an equational characterization of 
  $n$-core partitions, and the number of significant arrows of a monomial ideal 
  as defined in \cite{Maclagan_Smith_2010}.
\end{abstract}

\maketitle
\tableofcontents

\section*{Notation and Conventions}

Some notation which will be observed throughout this 
paper is as follows:

$\bullet$ All schemes and stacks will have $\C$ as their field of definition.

$\bullet$ 
$G$ will refer to a finite abelian group acting linearly on $\aff^2$. The 
notation $G$-$...$ will be a shorthand for $G$-equivariant. $\RR = \C[G]$ 
denotes the regular representation of $G$. We will identify $\Z/(n)$
with $\mu_n$, the group of $n$-th roots of unity, by fixing a choice of 
primitive root of unity $\zeta$ and using the association 
$i \leftrightarrow \zeta^i$. The $(n; a,b)$ action will refer to the 
action of $G = \mu_n$ given by $\zeta \cdot (x,y) = (\zeta^a x, \zeta^b y)$ 
for all $(x,y) \in \aff^2$.

$\bullet$ 
For $\rho_1, \rho_2 \in \Rep(G)$, the expression 
$\rho_1 + \rho_2$ will denote $\rho_1 \oplus \rho_2$ and 
$a \rho_1$ denotes $\rho_1^{\oplus a}$, and $- \rho$ denotes 
the formal inverse of addition of representations so that 
$\rho + (-\rho) = 0$.

$\bullet$ 
As $G$ is abelian, its irreducible representations are all one dimensional. 
As such $\Irr(G)$ will be identified with the set of characters of $G$ so 
that in particular $\Irr(G) \simeq G^*$. By fixing a choice of root of 
unity we may also form a isomorphism $\chi: G \simeq \Irr(G)$ 
which will be referred to as a \textbf{character isomorphism}. When 
we wish to emphasize the representation we will use $\chi(g)$,
and if instead we wish to emphasize the character then we will use 
$\chi_g$. In the event that $G = \mu_n$ and $\zeta$ is the given choice 
this will amount to letting $\chi(i)$ correspond to the character 
$\chi_i: G \rightarrow \C^*$ given by $\chi_i(j) = \zeta^{ij}$.

$\bullet$ 
Let $R$ be commutative ring and $G$ act on $R$ by ring automorphisms. 
$\fgMod_{R}^G$ denotes the category of finitely generated $R$-modules 
with an action of $G$ that is compatible with the module structure.

$\bullet$ 
A \textbf{planar abelian stack} is one of the form $\eex = [\aff^2/G]$ where 
$G$ acts linearly on $\aff^2$. For a $G$-subscheme $\xi \subset \aff^2$ (resp. 
$\zz \subset \eex$ a substack) define $H^0(\xi) = H^0(\xi, \Ocal_{\xi})$ 
(resp. $H^0(\zz) = H^0(\zz,\Ocal_{\zz}))$ and $h^0(\xi) = \dim_{\C}(H^0(\xi))$ 
(resp. $h^0(\zz) = \dim(H^0(\zz))$). 

$\bullet$
We will abuse notation and 
use $p \in \eex$ to denote a geometric point $p = \Spec(\C) \rightarrow \eex$,
and $\zz \subset \eex$ (resp. $\mathfrak{U} \subset \eex$) to denote a 
representable closed (resp. open) immersion $\zz \rightarrow \eex$ 
(resp. $\mathfrak{U} \rightarrow \eex$).
$\Delta \subset \eex$ will denote the open substack of $\eex$ 
which consists of geometric points with trivial stabilizer.

\section{Introduction}

The $G$-equivariant geometry of zero dimensional subschemes is a subject which 
has been studied extensively; See 
\cite{BejleriZaimi2023,CastroRoss2019,guseinzade2010generating,Haiman_Sturmfels_2002,Johnson2013Topology,Maclagan_Smith_2010}.
In this paper we will investigate these Hilbert schemes by means of the 
Euler form of the stack, a quadratic form which measures the growth rates 
of the dimensions of $G$-equivariant Hilbert schemes. Similar to the McKay 
correspondence when the group acts on $\aff^2$ by an $SL(2)$ action we will use 
the Euler form to construct a weighted graph and give a classification of 
abelian planar stacks. 

Let $\eex = [\aff^2/G]$ be an abelian planar stack. Following the work of 
Olsson and Starr in \cite{Olsson_Starr_2008}, for any $\rho \in \Rep(G)$ it is 
possible to associate to $\eex$ a quasiprojective scheme $\eex^{[\rho]}$ which 
acts as the fine moduli space parameterizing zero dimensional substacks 
$\zz \subset \eex$ such that if $\xi$ is the $G$-equivariant subscheme 
corresponding to $\zz$ then $H^0(\xi) \simeq \rho$ as representations. 

Section 2 collects various results which form the technical basis 
of the analysis through the entirety of the paper. Notably the 
main objects of study are defined and various reductions to aid in their 
study are proven.

Section 3 makes use of a Grothendieck-Riemann-Roch theorem for 
tame Deligne-Mumford (= DM) stacks to be able to calculate a 
quadratic form which may be associated to $\eex$, which will be 
referred to as the Euler form. Fix $\rho \in \Rep(G)$ and form its 
isotypic decomposition 
\begin{equation*}
  \rho \simeq \sum_{g \in G} a_g \chi(g)
\end{equation*}
where $a_g \in \N$. After diagonalizing the standard representation of $G$
define $g_x,g_y \in G$ by the isomorphisms
$\Span(x) \simeq \chi(g_x), \Span(y) \simeq \chi(g_y)$ for 
$x,y \in \C[x,y]$ and define the canonical element of the action to be 
$c = g_x + g_y$. The Euler form, denoted $Q_{\eex}(\{a_g\})$, is a quadratic 
form in the coefficients of the isotypic decomposition and is used to give 
formulas for the dimensions of tangent spaces of $\eex^{[\rho]}$.
We will prove that if $\zz \in \eex^{[\rho]}$ then 
\begin{equation}\label{DimEq}
  \dim(T_{\zz} \eex^{[\rho]}) = a_0 + a_{-c} - Q_{\eex}(\{a_g\})
\end{equation}
In particular as the right hand side of the equation depends only on 
$\rho$ this implies that the tangent spaces of $\eex^{[\rho]}$ have a 
constant dimension whence it is always smooth, a result consistent with 
Theorem 1.1 of \cite{Maclagan_Smith_2010}.

Section 4 consists giving explicit calculations of $Q_{\eex}$. More precisely 
let $G$ be a finite abelian group acting on $\aff^2$
such that the stabilizer group of the generic point is trivial. 
Fix $\rho \in \Rep(G)$ and let $a_g = \innprod{\rho}{\chi_g}$ for 
all $g \in G$ so that in particular $\rho \simeq \sum_g a_g \chi(g)$. Finally 
set $\rhostd = \chi_{g_x} + \chi_{g_y}$,  and let $D = (G^2 - \Delta)/S_2$
where $S_2$ acts by a transposition of the coordinates. We will prove Theorem 
\ref{EulerFormCoeffs} which states
\begin{theorem*}
  \begin{equation*}
    Q_{\eex}(\{a_g\}_{g \in G}) = \sum_{(g,h) \in D} b_{g,h}(a_g - a_h)^2
  \end{equation*} 
  where $b_{g,h} = \frac{1}{2} \innprod{\rhostd - \chi_c}{\chi_{g-h} + \chi_{h-g}}$
\end{theorem*}

Section 5 consists of the construction of a natural graph which may be 
associated to $\eex$. It will be referred to as the generalized McKay graph 
of $\eex$ and denoted $\Gamma_{\eex}$. A natural decomposition of the graph 
will be given so that 
$$\Gamma_{\eex} = \Gamma_{\eex}^{+} \cup \Gamma_{\eex}^{-}$$
for desirable spanning subgraphs. Moreso it will be shown if 
the associated homomorphism of the action $G \rightarrow GL(2)$ is 
injective then there is a natural isomorphism 
$\Gamma_G \simeq \Gamma_{\eex}^{+}$, where $\Gamma_G$ is the 
McKay graph of $G$. This will be used to reinterpret terms of equation 
\ref{DimEq} in terms of representation theoretic data and leave 
open the possibility for other contributions to originate from 
scheme theory.

Section 6 defines a grouping of elements $\{g_i\} \in G$ to be 
the subsets which consists of elements equal to each other. By Theorem 
\ref{EulerFormCoeffs} in order to give numerical values for the coefficients
it suffices to provide all of the possible groupings of
$\{g_x,-g_x,g_y,-g_y,c,-c\}$ arising from quotient stacks. Upon doing so the 
corresponding Euler form is calculated and \textit{a fortiori} the Euler form 
is calculated for all abelian planar stacks. 

Section 7 applies groupings to be able to classify all abelian planar stacks.
From there, analysis is conducted in individual cases to be able to more 
precisely determine which actions give rise to a stack with the given 
grouping. The ideas are then developed to re-interpret certain classes 
in terms of points and subsets of a ``moduli space" of all such quotient 
stacks.

The remainder of the paper will consist of an analysis of one of the 
families which arises from the previous classification: diagonal type stacks
for a cyclic action. This is done by an analysis of Young diagrams colored 
according to the rule $\CC_n(i,j) = i - j \modn$. From a scheme theoretic 
point of view, $\CC_n(i,j)$ corresponds to the isomorphism of representations
$\Span(x^i y^j) \simeq \chi(i - j)$.

Section 8 will discuss the connection between colored diagrams and certain 
elements of $\eex$, referred to as monomial substacks. This provides the 
necessary means to employ stacks to study Young diagrams. Next we will 
discuss algebraic means to represent removing border strips from a diagram. 
Four notions of minimal diagrams will be developed in the general case, 
referred to as $B$-core, $n$-core, $MS$-core, and $G$-core. We will show 
all four of these notions differ for general actions. 

In Section 9 we will show for all four definitions of core coincide for 
diagrams colored with $\CC_{n}$ by proving Theorem \ref{AllCoreEq}:
\begin{theorem*}
  Let $Y$ be a Young diagram colored with respect to $\CC_n$. Then the 
  partition associated to $Y$ is an $n$-core iff $Y$ contains no significant 
  arrows as defined in \cite{Maclagan_Smith_2010}.
\end{theorem*}
Furthermore by Theorem \ref{EulerFormCoeffs} of this note and Theorem 4.9 of 
\cite{Maclagan_Smith_2010} we will be able to give explicit formula 
for the number of significant arrows of a monomial ideal. 

Section 10 will examine consequences of Theorem \ref{AllCoreEq}. 
In what follows for a Young diagram $Y$ colored by $\CC_n$ set 
$a_i = \#(\CC_n^{-1}(i) \cap Y)$. We will prove an explicit formula in 
terms of the $\{a_i\}$ for the number of border strips which may be removed 
from $Y$ before obtaining an $n$-core in Corollary 
\ref{NumBorderStripFormula}: 
\begin{coro*}
  Let $Y \subset \N^2$ be a Young diagram colored by $\CC_n$ and 
  set $a_n = a_0$. Then 
  $$2B = 2a_0 - \sum_{i = 0}^{n - 1} (a_i - a_{i + 1})^2$$
  where $B$ is equal to the number of border strips which can be 
  removed from $Y$ before obtaining an $n$-core.
\end{coro*}
In addition we will use this and related work to give an equational 
characterization of $n$-core diagrams. 
\begin{coro*}
  $Y$ is an $n$-core diagram if and only if
  $2a_0 = \sum_{i = 0}^{n - 1} (a_i - a_{i + 1})^2$
\end{coro*} 
The final analysis will be giving a stack theoretic interpretation of 
these two corollaries. Specifically $n$-core diagrams will be shown to 
correspond to minimal $\eex^{[\rho]}$ from which all others are built; 
This gives rise to a ``pseudo"-stratification of $\eex^{[\rho]}$ in the event 
that $\eex^{[\rho]}$ is not minimal.

\subsection*{Acknowledgements}

The first author would like to thank many people. The first of which is 
his family as well of course as everyone who has joined our little 
cluster over the years. You all have continued to show support over the 
years when his interest in math was first beginning, all the way through 
the time of writing and presumably beyond this as well; It is deeply 
appreciated. He would also like to thank his older friends and 
the many others he has come to meet over the years. In addition gratitude 
goes to his roommates and friends who have been able to make graduate school 
more than just a place to get a degree, and to Manny and Depp.

Both authors would also like to thank Sudheer Shukla 
who provided the proof of Lemma \ref{SudheersLemma}; The lemma led to 
the technique used to calculate the Euler forms. In addition 
we would like to thank Leonel Medina for his help in putting together a 
program to run experiments on the coefficients of Euler forms. The 
results of the program directly led to the idea that explicit formulas 
for the coefficients, i.e., Theorem \ref{EulerFormCoeffs}, were possible 
for more general actions.

\section{Some Results on DM Stacks}

This section will consist of the necessary prerequisite results and 
definitions which will be used through the remainder of the 
paper. As most of the results do not require much machinery to prove or 
follow from standard techniques, many details will be omitted and some proofs 
will be sketched. The remainder of the details which are not supplied
can be found in the first authors forthcoming dissertation, from which this 
paper was developed.  

\subsection{The Gilbert Scheme of a Representation}

For a zero dimensional substack $\zz \subset \eex$ let $\xi$ denote the 
corresponding $G$-subscheme of $\aff^2$. Necessarily $h^0(\xi) < \infty$ 
so that when $H^0(\xi)$ is equipped with the induced action of $G$ it may be 
regarded as an element of $\Rep(G)$. 
\begin{definition}
    The representation type of both $\zz$ and $\xi$ is the isomorphism class
    $[H^0(\xi)] \in \Rep(G)$, and denoted $\varrho(\zz), \varrho(\xi)$
    respectively.
\end{definition}
To ensure representation type is a suitable notion to consider for families 
of subschemes, set $X = \aff^2$, let $T$ be an integral scheme, and let
$\Xi \subset X \times T$ be a $G$-flat family of zero dimensional 
subschemes. Furthermore, define $X_T = X \times_{\C} T$ and
$\eex_T = \eex \times_{\C} T$.
\begin{prop}\label{GFlatLimit}
 There exists a unique $\rho \in \Rep(G)$ such that 
 $\varrho(\Xi_t) \simeq \rho$ for all closed $t \in T$.
 \begin{proof}
  Note the claim is local on $T$ and hence we may assume $T = \Spec(A)$. 
  Let $I$ denote the ideal corresponding to $\Xi$ in $A[x,y]$.
  Next fix a choice of $t \in T$ and let $m_t \subset A$ 
  denote the corresponding maximal ideal. Pass to the 
  localization $M_t = M \otimes_A A_{(m_t)}$. From the definition of 
  $\rho = \varrho(\Xi_t)$ it follows there is an isomorphism of 
  representations $M/ m_t M \simeq \rho$. By Nakayamas lemma we may lift a 
  basis of $M/ m_t M$, $\{f_i'\}$, to a minimal generating set of $M_t$ denoted 
  as $\{f_i\}$. Clear denominators so that we may assume 
  $f_i \in \im(M \rightarrow M_t)$ and thus also assume $\tilde{f}_i$
  is a lift of $f_i$ for each $i$. Now consider 
  $\sum_i A \tilde{f}_i \subset M$. 
    
  As $A \rightarrow A_{(m_t)}$ is injective and $M$ is flat it follows 
  $M \rightarrow M_t$ is injective as well. In particular for a relation of 
  the form
  \begin{equation*}
    a_1 \tilde{f}_i + ... + a_m \tilde{f}_m = 0
  \end{equation*}
  upon pushing forward and using freeness of $M_t$ necessarily each 
  $a_i = 0$. In other words $\sum_i A \tilde{f}_i$ is a free $A$-module.
  Hence by comparing lengths and using $\bigoplus_i A_t f_i \simeq M_t$ 
  we may prove $M \simeq \bigoplus_i A \tilde{f}_i$.
  Now factor $A$ from the last direct sum so
  \begin{equation*}
    M \simeq \bigoplus_i A \tilde{f}_i \simeq A \otimes \bigg(\bigoplus_i \C \tilde{f}_i \bigg)
  \end{equation*}
  As $A \rightarrow A_{(m_t)}$ is an injective $\C$-linear map, the action of $G$ 
  fixes the coefficients. Each $f_i$ is a lift so that
  \begin{equation*}
    \bigoplus_i \C \tilde{f}_i \simeq \bigoplus_i \C f_i \simeq \bigoplus_i \C \tilde{f}_i' \simeq \rho \in \Rep(G)
  \end{equation*}
  Hence $M \simeq A \otimes_{\C} \rho$ for some $\C$-linear injection 
  $1 \otimes \rho \rightarrow M$. To finish up choose an arbitrary 
  maximal ideal $m_s \subset A$ so $A/m_s \simeq \C$ and calculate the 
  isomorphisms of $\C$-modules
  \begin{equation*}
    \varrho(\Xi_s) = M_s \simeq M \otimes_A A/m_s \simeq A/m_s \otimes_{\C} \rho \simeq \rho
  \end{equation*}
  For the last isomorphism we have used $A \rightarrow A/m_s$ is a 
  $\C$-algebra homomorphism and hence is injective on $\C$. Finally we applied 
  a simple dimension count.
 \end{proof}
\end{prop}
This proposition will be more succinctly referred to as ``$G$-flat limits 
preserve representation type". One useful consequence is a characterization 
of flat families of zero dimensional substacks of $\eex$, which enables us to 
translate between $G$-subschemes of $\aff^2$ and substacks of $\eex$.
\begin{prop}\label{FlatnessCrit}
    Fix an integral scheme $T \in Sch/k$. Define $G$ to act trivially on $T$ 
    in the product $\eex_T$. Let $\mathfrak{Z} \rightarrow \eex_T$
    be a family of 0 dimensional substacks and let $\Xi \subset \aff_T^2$
    be the associated family of subschemes.
    Then $\mathfrak{Z} \rightarrow T$ is flat iff there exists a 
    unique $\rho \in \Rep(G)$ such that $\varrho(\Xi_t) \simeq \rho$ for 
    all closed $t \in T$.
\end{prop}
The proof is straightforward, albeit lengthy. First a reduction is made to 
the case when $T = \Spec(A)$ is performed. The key point is that
a flat family of substacks, $\mathfrak{Z}$, corresponds to a family of modules 
$M$ over $A$ via the equivalence of categories 
$\Coh(\eex_T) \simeq \fgMod_{A[x,y]}^G$. As $G$ acts trivially on $A$ and $M$ is a
finite module over $T$, it admits an isotypic decomposition 
$M \simeq \bigoplus_{\chi \in \Irr(G)} M_{\chi}$. The proposition follows from:
\begin{lemma} $\text{$\mathfrak{Z}$ is flat over $T$} \Leftrightarrow \text{$M$ is flat} \Leftrightarrow \text{$M_{\chi}$ is flat for all } \chi$
\end{lemma}
Next for any $\rho \in \Rep(G)$ the functor 
$\Hilb_{\rho}: (Sch/\C_{et})^{op} \rightarrow Set$
is defined as follows: For $T \in Sch/\C$, define $\Hilb_{\rho}[\eex](T)$ to 
consist of all flat families of zero dimensional substacks 
$\XX \subset \eex_T$ such that $\varrho(\XX_t) \simeq \rho$ for all closed 
points $t \in T$. The previous two results show $\Hilb_{\rho}$ is well defined
and just as well also parameterizes families of subschemes with the 
given representation type. Following the results of \cite{Olsson_Starr_2008}, 
in a forthcoming work it will be proven
\begin{prop}
    $\Hilb_{\rho}(\eex)$ is represented by a quasiprojective scheme. 
\end{prop}
The fine moduli space of $\Hilb_{\rho}(\eex)$ will be denoted $\eex^{[\rho]}$ 
and called the $G$-equivariant Hilbert scheme with respect to $\rho$, or 
more frequently as a Gilbert scheme of $\eex$.
\begin{remark}
  The proposition holds for all $k = \overline{k}$, finite and linearly 
  reductive $G$, and smooth quasiprojective $X$. In addition, every result of 
  this section has an analogue when $\aff^2$ is replaced either by a smooth 
  quasiprojective $X$ or by $\aff^n$ when the replacement is valid for the 
  result at hand.
\end{remark}

\subsection{A Comparison with Multigraded Hilbert Schemes}

There is a connection between Gilbert Schemes and 
Multigraded Hilbert Schemes as defined by Haiman and Sturmfels in 
\cite{Haiman_Sturmfels_2002}. Let $G$ act linearly on $\aff^2$
and denote the homogeneous linear polynomials by $\C[x,y]_1$.
It is easy to show there exists $x,y \in \C[x,y]_1 \simeq \rhostd$ which 
diagonalize the standard representation; These coordinates will be referred 
to as \emph{isotypic coordinates}. 
We will now recall some definitions from (loc. cit.) which will
prove to be useful later. To delineate between definitions and conventions 
some modifications to the existing notation will be employed:

$\bullet$ 
The collection of monomial ideals will be identified with $\N^2$ by the 
association $x^i y^j \leftrightarrow (i,j)$. A $G$-grading on $\C[x,y]$
will refer to a semigroup homomorphism $\deg: \N^2 \rightarrow G$ 
and associated grading $\deg_G(x^i y^j) = \deg(i,j)$. 
Given a $G$-grading there is an associated homogeneous decomposition 
\begin{equation*}
  \C[x,y] \simeq \bigoplus_{g \in G} \C[x,y]_g
\end{equation*}
where $\C[x,y]_g$ is the span of monomials $x^i y^j$ with 
$\deg_G(x^i y^j) = g$; Furthemore 
$\C[x,y]_g \cdot \C[x,y]_h \subset \C[x,y]_{gh}$. For a $\C$-algebra $A$ we 
may extend $\deg_G$ to $A[x,y]$ by setting $\deg_G(a) = 0$ for all $a \in A$.

$\bullet$ 
For a $\C$-algebra $A$, an ideal $I \subset A[x,y]$ which is homogeneous with 
respect to $\deg_G$ will be called admissible if for all $g \in G$ we have 
$\text{rank}(A[x,y]_g/I_g) < \infty$. The HS-polynomial (referred to as 
Hilbert polynomial in (loc. cit.)) of an admissible ideal is defined by 
$h_I(g) = \text{rank}(A[x,y]_g/I_g)$. 

$\bullet$ 
For a given function $h: G \rightarrow \N$, the 
sheaf $H_G^{h}: (\C-\text{alg})^{op} \rightarrow Set$ is defined by 
associating to a $\C$-algebra $A$ the collection of admissible ideals 
$I \subset A[x,y]$ such that $h_I = h$. By Theorem 1.1 of (loc. cit.), 
$H_G^h$ is represented by a quasiprojective scheme which we will denote by 
$\Hilbs_G^h$.

\medskip
Now that the necessary definitions have been discussed we will 
continue with the comparison to Gilbert schemes.
Fix a character isomorphism $\chi: G \simeq \Irr(G)$.
Given the action and isotypic coordinates it follows from the definitions 
that $\Span(m)$ is a $G$-stable subspace for each monomial $m \in \C[x,y]$ 
and hence there exists $g \in G$ such that $\Span(m) \simeq \chi(g)$. Next 
define the elements $g_x, g_y \in G$ by 
$\Span(x) \simeq \chi(g_x), \Span(y) \simeq \chi(g_y)$ for $x,y \in \C[x,y]$.
Define a $G$-grading by setting $\deg(x) = g_x, \deg(y) = g_y$ and 
extending this to the semigroup homomorphism $\deg_G: \N^2 \rightarrow G$.

By direct verification it follows that HS-polynomials are in a bijective 
correspondence with $\Rep(G)$ since the only degrees of freedom between both 
are the multiplicities of irreducible representations. 
Then the primary mechanism to translate between Gilbert Schemes and 
Multigraded Hilbert schemes is given in the following proposition. 
\begin{prop}
  Let $\rho \in \Rep(G)$ and $h_{\rho}$ be the corresponding Multigraded 
  Hilbert Polynomial. Then there is a natural isomorphism of 
  schemes $\eex^{[\rho]} \simeq \Hilbs_G^{h_{\rho}}$.
  \begin{proof}
    To begin, the functor $H_G^h$ is defined on $\C$-algebras which by 
    the $\Spec$ construction is equivalent to the category of affine 
    $\C$-schemes. As affine schemes form a base of the Zariski site 
    of $Sch/\C$ it follows $H_G^{h_{\rho}}$ may be extended to a sheaf; 
    $H_G^{h_{\rho}}$ will also be used to refer to the extension. 
    Clearly $\Hilbs_G^{h_{\rho}}$ still represents the extension. 
    As such to prove the claim it suffices to show there is a natural 
    isomorphism $\Hilb_{\rho}(T) \simeq H_G^{h_{\rho}}(T)$
    for an arbitrary affine $\C$-scheme $T = \Spec(A)$. Moreso as both 
    $\eex^{[\rho]}, \Hilbs_G^{h_{\rho}}$ are quasiprojective 
    schemes it also suffices to only consider when $A$ is noetherian.
    
    For a given morphism $f: T \to \Hilbs_G^{h_{\rho}}$ let $I$ denote the 
    family of ideals arising from the universal property of fine moduli spaces.
    Next form the isotypic decomposition $\rho = \sum_{g \in G} a_g \chi(g)$. 
    From the definitions we have 
    \begin{equation*}
      A[x,y]/I \simeq \bigoplus_{g \in G} A[x,y]_g/I_{g}
    \end{equation*}
    The assertion $h_{\rho}$ is constant on $\Spec(\C)$ when 
    extended to families amounts to the requirement 
    $$h_{\rho}((A[x,y]/I)_{g} \otimes_A m ) = a_g$$ 
    for all maximal $m \subset A$. Importantly as $G$ acts trivially on $A$ 
    it follows $\deg(a) = 0$ for all scalars so the isomorphism
    $(A[x,y]/I) \otimes_A m \simeq (A/m)[x,y]/I$
    is compatible with the group structure and the $G$-grading. Thus 
    there is an isomorphism of graded rings 
    \begin{equation*}
      (A[x,y]/I) \otimes_A m \simeq \bigoplus_{g \in G} (A/m)[x,y]_g/I_{g}
    \end{equation*}
    In total, the constancy of $h_{\rho}$ is equivalent to 
    requiring the length of $(A/m)[x,y]_g/I_{g}$ to be independent
    of the choice of maximal ideal $m$. 
    
    Set $\Xi = \Spec(A[x,y]/I)$
    and let $\mathfrak{Z}$ denote the corresponding family of stacks. 
    It is easy to verify that 
    being an admissible ideal implies $I$ is a $G$-ideal 
    with respect to the given action. Moreso as $\sum_i a_i$
    is equal to both the colength of $I_t$ and $\deg(\rho)$ it follows each 
    fiber of $I$ is a finite colength ideal. As $T$ is an integral 
    noetherian scheme we may use that $\Xi$ has constant representation 
    type and Proposition \ref{FlatnessCrit} to show 
    $\mathfrak{Z}$ is flat over $T$. As this is compatible with 
    base change the association $I \to \mathfrak{Z}$ defines a 
    functor of sheaves $H_G^{h_{\rho}} \rightarrow \Hilb_{\rho}$. 
    To finish up it suffices to note for a given flat family of substacks 
    $\mathfrak{Z}$ by first forming the associated family of schemes $\Xi$ 
    that $I(\Xi)$ is an admissible ideal with respect to the given 
    $G$-grading and this association is unique. As such the reverse 
    association $\mathfrak{Z} \rightarrow I$ defines an inverse morphism of 
    sheaves to the one constructed above. Finally by Yoneda's lemma there is 
    an induced isomorphism on the fine moduli spaces which is what we wished 
    to prove.
  \end{proof}
\end{prop}
As such we may freely use results pertaining to Multigraded 
Hilbert schemes, and most notably the lovely paper
\cite{Maclagan_Smith_2010}. In particular we obtain the following consequence 
of Theorem 1.1:
\begin{coro}
  $\eex^{[\rho]}$ is a smooth quasiprojective variety all $\rho \in \Rep(G)$.
\end{coro}

\subsection{Rigidification}

Let $\eex = [\aff^2/G]$, and let $H < G$ be the stabilizer group of the 
generic point of $\aff^2$. Recall that $\Irr(G) \simeq G^*$ so that by 
dualizing the canonical projection $G \rightarrow G/H$ there is an 
induced 
natural inclusion $\Irr(G/H) \subset \Irr(G)$. This may be extended 
linearly 
to define an inclusion $\Rep(G/H) \subset \Rep(G)$ in the obvious way. 
Next by regarding $H$ as a group scheme over $\Spec(\C)$ we may perform a 
rigidification of $\eex$; See section 5 of 
\cite{Abramovich2003}. Upon doing so 
it is easy to verify the stack $\eexgr = [\aff^2/(G/H)]$ is the rigidification 
of $\eex$. We will say $\eex$ is \textbf{rigid} if $\eex = \eexgr$.
For example if $\eex = [\aff^2/\mu_4]$ with the $(4;2,2)$ action then 
$\eexgr = [\aff^2/\mu_2]$ with the $(2;1,1)$ action.
There is always a direct comparison between a stack and its rigidification which may be performed, as in 
\cite{Abramovich2003}. As this is 
a simple case we may explicitly describe the morphism.
\begin{lemma}
    There is a natural morphism of stacks $\eex \rightarrow \eexgr$ which 
    is a homeomorphism on geometric points. 
    \begin{proof}
        As $\eexgr$ is a rigidification of $\eex$ included in this data 
        is a functorial morphism $F: \eex \rightarrow \eexgr$. More 
        concretely we will briefly describe its construction.
        To begin, let $T \in Sch/k$. As $G$ is a finite group it is 
        affine so in particular the fiber category $\eex(T)$ is equivalent to 
        the groupoid of principal $G$-bundles 
        \[\begin{tikzcd}
        	P & X \\
        	T
        	\arrow["\pi", from=1-1, to=1-2]
        	\arrow[from=1-1, to=2-1]
        \end{tikzcd}\]
        where $\pi$ is a $G$-equivariant morphism; Further discussion of this 
        may be found on pages 171-172 of \cite{Olsson_2016}. As $H$ acts 
        trivially with respect to the action on $\aff^2$ it is clear the 
        action on families is also trivial. As such we may meaningfully define 
        an action of $G/H$ on principal bundles by the declaration
        $gH \cdot P = g \cdot P$. From here it is easy to see $\pi$ is 
        also a $(G/H)$-equivariant morphism and $F$ is defined by 
        regarding $P$ as a $(G/H)$-bundle with the constructed action.

        Next let $S$ be the coarse moduli space of $\eex$ and $S_{\operatorname{rig}}$ the 
        coarse moduli space of $\eexgr$. As $G$ is finite and $\aff^2$ is 
        quasiprojective we have $S = \aff^2/G$ and $S_{\operatorname{rig}} = \aff^2/(G/H)$
        with their respective actions. Similar to the above argument there 
        is a morphism $f: S \rightarrow S_{\operatorname{rig}}$ and in fact $f$ is an 
        isomorphism. As there is a homeomorphism 
        between the geometric points of a DM stack and its coarse moduli 
        space the final claim follows from the commutativity of 
        \[\begin{tikzcd}
        	\eex & {\eexgr} \\
        	S & {S_{\operatorname{rig}}}
        	\arrow["F", from=1-1, to=1-2]
        	\arrow["q"', from=1-1, to=2-1]
        	\arrow["{q_{\operatorname{rig}}}", from=1-2, to=2-2]
        	\arrow["f"', from=2-1, to=2-2]
        \end{tikzcd}\]
        and the composition of homeomorphisms            
        $\eex \rightarrow S \xrightarrow{f} S_{\operatorname{rig}} \rightarrow \eexgr$.
    \end{proof}
\end{lemma}
Next suppose $\zz \subset \eex$ is a zero dimensional substack with 
associated $\xi \subset \aff^2$. With the above action we may 
regard $\xi$ as a $(G/H)$-subscheme. Moreso as the action of 
$H$ on $\C[x,y]$ is by trivial ring automorphisms we have
$\varrho(\zz) = \varrho(\xi) \in \Rep(G/H) \subset \Rep(G)$. As such if 
$\eex^{[\rho]} \neq \emptyset$ then necessarily $\rho \in \Rep(G/H)$.
The main result of this subsection is:
\begin{prop}
For all $\rho \in \Rep(G/H)$ there is an isomorphism
$P_{\operatorname{rig}}: \eex^{[\rho]} \simeq \eexgr^{[\rho]}$.
\begin{proof}
    More concretely, $P_{\operatorname{rig}}$ is given on $T$-valued 
    points $P \to T$ by regarding $P$ as a $(G/H)$-bundle in the 
    same manner as described above. An inverse can be 
    explicitly described. By the quotient $G \rightarrow G/H$, if $P$ is a 
    $(G/H)$-bundle under the constructed action clearly it may also 
    be regarded as a $G$-bundle. Moreso its representation type may be 
    calculated either with respect to $G$ or $G/H$, and this 
    association is functorial. Thus by a functor of points argument 
    there is an induced morphism 
    $Q_{\operatorname{rig}}: \eexgr^{[\rho]} \rightarrow \eex^{[\rho]}$. 
    To complete the proof we need only note from their constructions clearly 
    $P_{\operatorname{rig}} \circ Q_{\operatorname{rig}}$ and 
    $Q_{\operatorname{rig}} \circ P_{\operatorname{rig}}$ are the 
    identity morphism on their respective spaces as the underlying bundle 
    is unchanged by both maps.
\end{proof}
\end{prop}
If $\eex$ is rigid then necessarily $\Delta \neq \emptyset$ and in fact is a 
dense open substack. Recall if $\zz \subset \eex$ is a substack and 
$\xi \subset \aff^2$ is the corresponding subscheme then $\varrho(\zz)$ 
is the isomorphism class of $H^0(\xi)$ as a representation, and $\RR = \C[G]$.
\begin{prop}\label{DeltaRepCase}
  Suppose $\zz \subset \Delta$ is a zero dimensional substack. Then 
  $\varrho(\zz) \simeq m \RR$ for some $m \in \N$. 
  \begin{proof}
    Let $\xi \subset \aff^2$ denote the $G$-subscheme corresponding to $\zz$. 
    $\xi_{red}$ decomposes into a disjoint union of free and reduced 
    orbits, $p_i = \Orb_G(x_i)$, for some $x_i$. Let $\xi_i$ denote the 
    maximal subscheme of $\xi$ supported on $p_i$ so 
    $\xi = \bigsqcup_i \xi_i$. As the union is disjoint there is a 
    $G$-isomorphism
    \begin{equation*}
      H^0(\xi) \simeq \bigoplus_i H^0(\xi_i)
    \end{equation*}
    from which the claim follows from the special case of each summand. 
    Hence without loss of generality assume $\xi_{red} = \Orb_G(q)$ for some 
    reduced $q \in \aff^2$. Now consider the $G$-isomorphism 
    \begin{equation*}
      H^0(\xi) \simeq \bigoplus_{g \in G} H^0(g \cdot q)
    \end{equation*}
    There are $\# G$ distinct summands since q has a trivial stabilizer.
    If we set $m = h^0(q)$ this implies $h^0(\xi) = m \cdot \#G$.
    Next let $\{f_i\}_{i = 1}^m$ be a basis of $H^0(\xi)$ as a 
    k-vector space and define 
    $$R_i = \bigoplus_{g \in G} \Span(g \cdot f_i)$$ 
    As $g$ acts transitively on $\xi_{red} = Orb_G(q)$, each $R_i$ consists of
    $\# G$ distinct summands. 
    By the association $f_i \leftrightarrow 1 \in \RR$, a 
    dimension count, and comparison of the actions there is an 
    isomorphism of representations $R_i \simeq m_i \RR$ 
    for some $m_i \in \N$. As
    \begin{equation*}
      H^0(\xi) \simeq \bigoplus_{i = 1}^{m} R_i
    \end{equation*}
    and the inclusion of each summand is a $G$-linear map this 
    implies $H^0(\xi) \simeq m \RR$ as representations for 
    $m = \sum_i m_i$.
  \end{proof}
\end{prop}
Now assume $\Delta \neq \emptyset$ so that it constitutes a dense open 
substack of $\eex$. It can be verified that since $\Delta$ consists of 
geometric points with trivial stabilizer group that the restriction of the 
coarse moduli space $\pi: \eex \rightarrow S$ induces an isomorphism 
$\Delta \simeq \delta = \pi(\Delta)$. Thus there is an isomorphism of schemes 
$\Delta^{[m \RR]} \simeq \delta^{[m]}$, and this will be employed to give more
refined topological descriptions of $\eex^{[\rho]}$; See Section 
\ref{TopGilbSchemes}.

\subsection{Equivalence of Actions}

In order to simplify the analysis we will define an equivalence relation on 
actions. Let $G$ be a finite abelian group acting on $\aff^2$ by 
means of the homomorphisms $\sigma_1, \sigma_2: G \rightarrow GL(2)$, and 
let $\eex_i$ be the quotient stack corresponding to $\sigma_i$.
\begin{definition} 
  We will say $\sigma_1$ is equivalent to $\sigma_2$, denoted 
  $\sigma_1 \simeq \sigma_2$ if there exists an
  isomorphism $(\eex_1)_{\operatorname{rig}} \simeq (\eex_2)_{\operatorname{rig}}$. 
\end{definition}
One interest in equivalent actions comes from the following result.
\begin{prop}
  Let $\sigma_1, \sigma_2: G \rightarrow GL(2)$ be two actions 
  with corresponding quotient stacks $\eex_1, \eex_2$ such that there is 
  a representable isomorphism $(\eex_1)_{\operatorname{rig}} \simeq (\eex_2)_{\operatorname{rig}}$. Then for 
  all $\rho \in \Rep(G)$ there exists a natural isomorphism 
  $\eex_1^{[\rho]} \simeq \eex_2^{[\rho]}$.
  \begin{proof}
    To begin as $\sigma_1 \simeq \sigma_2$ there is a representable
    isomorphism $(\eex_1)_{\operatorname{rig}} \simeq (\eex_2)_{\operatorname{rig}}$ so for the sake 
    of simplicity there is a commutative diagram
    \[\begin{tikzcd}
      {\aff^2} & {\aff^2} \\
      {(\eex_1)_{\operatorname{rig}}} & {(\eex_{2})_{\operatorname{rig}}}
      \arrow["f", from=1-1, to=1-2]
      \arrow[from=1-1, to=2-1]
      \arrow[from=1-2, to=2-2]
      \arrow["F", from=2-1, to=2-2]
    \end{tikzcd}\]
    and in particular $f$ is a $G$-isomorphism; The following argument 
    is also valid if $\aff^2$ is replaced with different etale 
    covers of the rigidification. As such if 
    $\xi \in (\eex_1)_{\operatorname{rig}}^{[\rho]}$ then by pullback there is an 
    induced $G$-isomorphism $H^0(\xi) \simeq H^0(f(\xi))$
    and hence $\varrho(\xi) \simeq \varrho(f(\xi))$. As such we 
    obtain a morphism 
    $(\eex_1)_{\operatorname{rig}}^{[\rho]} \rightarrow (\eex_2)_{\operatorname{rig}}^{[\rho]}$
    which admits an inverse as $F$ is an isomorphism. The claim 
    then follows from the chain of isomorphisms 
    $\eex_1^{[\rho]} \simeq (\eex_1)_{\operatorname{rig}}^{[\rho]} \simeq (\eex_2)_{\operatorname{rig}}^{[\rho]} \simeq \eex_2^{[\rho]}$.
  \end{proof}
\end{prop}
On the other hand other interesting examples arise when the 
isomorphism is not representable.
\begin{prop}\label{AutoRemoval}
  Let $\sigma_1, \sigma_2: G \rightarrow GL(2)$ be two homomorphisms 
  and assume there exists an automorphism $\varphi \in \operatorname{Aut}(G)$
  such that $\sigma_2 = \sigma_1 \circ \varphi$. Then 
  $\sigma_1 \simeq \sigma_2$.
  \begin{proof}
    The key point is that $G$ is finite so that any automorphism 
    is given by a scheme theoretic isomorphism. More precisely 
    as $G$ is finite it may be represented by the group scheme 
    $\bigsqcup_{g \in G} C_g$ where $C_g \simeq \Spec(\C)$ for each 
    $g \in G$ and the group action is given in the usual way. 
    As each automorphism is uniquely determined as a permutation 
    of the components this shows $\varphi$ gives rise to an automorphism 
    of group schemes. Then the action on principal bundles is defined 
    according to the product $G \times P \rightarrow P$ so that $\varphi$
    induces an automorphism of principal bundles. Continuing on in 
    this fashion we may construct a morphism of fiber categories 
    $\eex_1(T) \rightarrow \eex_2(T)$ which evidently is compatible 
    with base change, and thus we obtain a morphism 
    $\eex_1 \rightarrow \eex_2$. As $\varphi$ is an automorphism in 
    turn the map of stacks admits an inverse and thus $\eex_1 \simeq \eex_2$
    from which it is trivial that $\sigma_1 \simeq \sigma_2$.
  \end{proof}
\end{prop}
In the case of cyclic actions the preceding proposition has the 
following consequence. In what follows for a given $(n;a,b)$ action assume 
$a, b \neq 0 \modn$. 
\begin{lemma}\label{UnitRemoval}
    For $u \in \Z/(n)^{\times}$, let 
    $\eex$ correspond to the $(n;a,b)$ action and $\eex'$ correspond 
    to the $(n;ua,ub)$ action. Then $\eex \simeq \eex'$.
\end{lemma}
For the next simplification we will give a means to numerically perform a 
rigidification. 
\begin{lemma}\label{DivisorRemoval}
    Let $\eex$ be the quotient stack corresponding to the $(dn;da,db)$ 
    action. Furthermore assume that $\gcd(a,b,n) = 1$. Then $\eexgr$ 
    corresponds to the $(n;a,b)$ action.
    \begin{proof}
    As $\gcd(a,b,n) = 1$ by examination of the proof that rigidifications
    exist we can clearly see $H = (d) \subset \mu_{dn}$.
    \end{proof}
\end{lemma}
For more general actions, as $G$ is a finite abelian group it is the 
case $G \simeq \bigoplus_{i = 1}^k \Z/(d_i)$ where $d_i \mid d_{i + 1}$
for all admissible $i$. 
\begin{definition}
  Let $e_i$ denote a fixed choice of generator of the $i$-th summand of 
  $G$ as above and set $m_i = d_{i}/d_{i-1}$. The 
  $$(d_1; f_{1,1}, f_{2,1}) \oplus ... \oplus (d_k; f_{1,k}, f_{2,k})$$
  action refers to the action of $G$ given by 
  $e_i \cdot (x,y) = (\zeta_{d_i}^{f_{1,i}} x, \zeta_{d_i}^{f_{2,i}} y)$
  such that $\zeta_{d_i} = \zeta_{d_{i + 1}}^{m_{i + 1}}$.
\end{definition}
The last condition of the definition is strictly stronger than having 
$G$ act by diagonal matrices. In order to ensure actions of this form are 
sufficiently general we will prove:
\begin{lemma}\label{DiagEq}
  Let $\sigma: G \rightarrow GL(2)$ denote an arbitrary representation with 
  an induced action on $\aff^2$. Then there exists $f_{j,i}$ for $j = 1, 2$,
  $i = 1,...,k$ such that 
  \begin{equation*}
    \sigma \simeq (d_1;f_{1,1}, f_{2,1}) \oplus ... \oplus (d_k; f_{1,k}, f_{2,k})
  \end{equation*} 
  \begin{proof}
    Set $\sigma_i: \Z/(d_i) \subset G \rightarrow GL(2)$ where the 
    first arrow is given b $1 \rightarrow e_i$. As isotypic coordinates 
    exist since $G$ is abelian we may simultaneously diagonalize each 
    $\sigma_i$ individually so that 
    $\sigma_i \simeq (d_1;f_{1,i}', f_{2,i}')$ for some parameters. To 
    show the final condition of the definition is met we will argue on 
    induction on $k$. Clearly if $k = 1$ then the claim is true. 
    Now assume we have shown it for all groups with less than $k$ 
    summands, and set $G' \simeq \bigoplus_{i = 1}^{k - 1} \Z/(d_i)$
    with the usual inclusion $G' \subset G$. If $\sigma'$ is 
    the induced representation then by the inductive hypothesis 
    necessarily
    \begin{equation}
      \sigma' \simeq (d_1;f_{1,1}, f_{1,2}) \oplus ... \oplus (d_{k - 1}; f_{1,k-1}, f_{2,k-1})
    \end{equation}
    for some choice of parameters and associated primitive roots of unity 
    $\zeta_{d_i}$. 
    
    On the other hand assume $(d_k;f_{1,k}', f_{2,k}')$
    is given by $e_k \cdot (x,y) = (\omega^{f_{1,k}'} x, \omega^{f_{2,k}'}y)$
    where $\omega$ is a primitive $d_k$-th root of unity. As 
    $\zeta_{d_{k - 1}}$ is itself primitive necessarily there exists 
    primitive $\zeta_{d_{k}}$ such that 
    $\zeta_{d_{k}}^{d_{k}/d_{k - 1}} = \zeta_{d_{k - 1}}$. Now observe 
    there exists a unit $u \in \Z/(d_k)$ such that $\omega = \zeta_{d_k}^u$,
    and set $f_{k,i} = u f_{k,i}'$. At each step we have only applied 
    automorphisms of $G$ in such a way that we obtain equivalent 
    actions. Thus by Lemma \ref{AutoRemoval} we obtain an 
    isomorphism of the corresponding quotient stacks at each step and hence 
    on induction we have proven
    $$\sigma \simeq (d_1;f_{1,1}, f_{2,1}) \oplus ... \oplus (d_k; f_{1,k}, f_{2,k})$$
  \end{proof}
\end{lemma}
An additional lemma will be required before stating the desired reduction for 
this paper and main result of this subsection.
\begin{lemma}
    Suppose $G \simeq \bigoplus_{i = 1}^{k} \Z/(d_i)$ with 
    $d_1 \neq 0$ and $d_i \mid d_{i + 1}$ for all $i$. Furthermore 
    assume $G$ acts linearly on $\aff^2$ and let $H \subset G$ denote the 
    stabilizer group of the generic point. If $k > 2$ then $H \neq 0$. 
    \begin{proof}
    Let $\sigma: G \rightarrow GL(2)$ denote the representation 
    corresponding to the given action. By the previous lemma we know 
    \begin{equation*}
      \sigma \simeq (d_1;f_{1,1},f_{2,1}) \oplus ... \oplus (d_k; f_{1,k}, f_{2,k})
    \end{equation*}
    for some choice of parameters. As we have in fact proven the 
    corresponding stacks are isomorphic we may prove the claim with the 
    latter action. As such we will calculate 
    \begin{equation*}
      (r_1 e_1 + ... + r_k e_k) \cdot (x, y) = (\zeta^{\frac{d_k}{d_1} r_1 p_{1,1} + ... + r_k p_{1,k}} x, \zeta^{\frac{d_k}{d_1} r_1 p_{2,1} + ... + r_k p_{2,k}} y)
    \end{equation*}
    The given element stabilizes $(x,y) \in \aff^2$ iff both exponents 
    are zero modulo $d_k$. As such elements of $H$ are in a bijective 
    correspondence with the kernel of 
    \begin{equation*}
      \begin{pmatrix}
        \frac{d_k}{d_1} p_{1,1} & ... & p_{1,k}\\
        \frac{d_k}{d_1} p_{2,1} & ... & p_{2,k}\\ 
      \end{pmatrix}
      \in \text{Hom}_{\Z}(\Z/(d_k)^{\oplus k},\Z/(d_k)^{\oplus 2})
    \end{equation*}
    As $k \geq 3$ by a counting argument the kernel must be nontrivial so 
    that $H \neq 0$.
    \end{proof}
\end{lemma}
The main result of this section is then given by 
\begin{prop}\label{CycDoubGenRed}
  Let $\sigma: G \rightarrow GL(2)$ be an arbitrary representation 
  such that associated action on $\aff^2$ is not the trivial action. Then 
  either $\sigma \simeq (n; a,b)$ or 
  $\sigma \simeq (m; a,b) \oplus (n;c,d)$ for some choice of parameters.
  \begin{proof}
    By Lemma \ref{DiagEq} we have 
    \begin{equation*}
      \sigma \simeq (d_1;f_{1,1}, f_{2,1}) \oplus ... \oplus (d_k; f_{1,k}, f_{2,k})
    \end{equation*}
    for some choice of parameters. By the previous lemma if $k > 2$ then 
    necessarily some element of $G$ stabilizes $\aff^2$. Thus we may pass to 
    the rigidification $[\aff^2/(G/H)]$. As an additional consequence of the 
    previous lemma necessarily $G/H$ is generated by at most two summands in 
    Smith normal form. Moreso as $G$ does not have the trivial action 
    we have $G \neq H$ and hence $G/H \neq 0$. If $G/H$ is cyclic then the 
    first claim holds. If $G/H$ is generated by two elements the second 
    claim holds.
  \end{proof}
\end{prop}
Thus in order to study the Gilbert schemes of abelian planar stacks 
it instead suffices to only consider actions of the form
$(n;a,b)$ or $(m;a,b) \oplus (n;c,d)$. As both cases will be 
analyzed we will say actions of the latter form are 
\textbf{doubly generated}.

\section{Tangent Spaces of Gilbert Schemes}

The principal focus of this section is to collect various results on 
dimensions of tangent spaces of $\eex^{[\rho]}$. The starting point 
for these considerations is a more general calculation. Throughout this 
section $\eex$ is an abelian planar stack.

\subsection{Grothendieck-Riemann-Roch For DM Stacks}

An intersection theory has been developed by Vistoli in \cite{Vistoli_1989}
for algebraic stacks of finite type over a field which has enabled the 
formulation of a Grothendieck-Riemann-Roch theorem for DM stacks in 
\cite{Toen_1998}. Recall the inertia stack is defined by the fiber square
\[\begin{tikzcd}
	{I \eex} & \eex \\
	\eex & {\eex \times_{\Spec(\C)} \eex}
	\arrow["q", from=1-1, to=1-2]
	\arrow[from=1-1, to=2-1]
	\arrow["\Delta", from=1-2, to=2-2]
	\arrow["\Delta", from=2-1, to=2-2]
\end{tikzcd}\]
For the sake of completeness we will recall some definitions which will be 
used in what follows. The following are taken primarily from (loc. cit.)
although some are standard within the field and will not attributed to a 
specific source. 

As $\eex$ is a tame quotient stack of $\aff^2$ it has the resolution 
property, that is every coherent sheaf admits a finite locally free 
resolution. Let $K(\eex)$ denote the Grothendieck group of locally free 
sheaves, or equivalently of coherent sheaves by the resolution property. 
We may decompose a locally free sheaf $V$ over $I \eex$ into 
eigensheaves $V \simeq \bigoplus_{\lambda \in \mu_{\infty}} V^{(\lambda)}$,
where $\mu_{\infty}$ denotes the algebraic group of roots of unity over $\C$
and $V^{(\lambda)}$ is the eigen-subsheaf whose eigenvalue is equal to 
$\lambda$. Next set 
$K(I \eex)_{\Lambda} = K(I \eex) \otimes_{\Q} \Q[\mu_{\infty}]$
and define the operation 
$\rho: K(I \eex) \rightarrow K(I \eex)_{\Lambda}$ on vector bundles by
\begin{equation*}
  \rho([V]) = \sum_{\zeta \in \mu_{\infty}} \zeta \cdot [V^{(\zeta)}]
\end{equation*}
and extending linearly; We will refer to this as the \emph{eigenvalue map}.
\begin{definition}
  The augmented Chern character is the composition of operators
  \begin{equation*}
    K(\eex)_{\Lambda} \xrightarrow{q^*} K(I \eex)_{\Lambda} 
    \xrightarrow{\rho} K(I \eex)_{\Lambda} \xrightarrow{\operatorname{Ch}} H^*(I \eex)
  \end{equation*}
  where $\operatorname{Ch}$ is the Chern character on the etale site of $I \eex$ and 
  is defined in a similar way to the usual Chern character.
\end{definition}
The following definitions originate from (loc. cit.), but the form which 
will be used comes from Appendix A of \cite{Tseng_2010}. 

The operator $\lambda_{-1}: K(I \eex) \rightarrow K(I \eex)$ is given on 
locally free sheaves by the equation 
\begin{equation*}
  \lambda_{-1}([V]) = \sum_{i = 0}^{\infty} (-1)^i [\Lambda^i V]
\end{equation*}
where $\Lambda^i V$ is the $i$-th exterior power of $V$, and extended 
linearly to $K(I\eex)$. For a locally free sheaf $V$ define the moving part by
\begin{equation*}
  V^{mov} = \bigoplus_{1 \neq \zeta \in \mu_{\infty}} V^{(\zeta)}
\end{equation*}
and the invariant part is $V^{inv} = V^{(1)}$. 
\begin{definition}
  The augmented Todd class of $\eex$ is given by
  \begin{equation*}
    \Td(V) = \frac{\operatorname{Td}((q^*V)^{inv})}{\operatorname{Ch}\big(\rho \circ \lambda_{-1}(\big[(q^*V)^{mov}\big]^{\vee}) \big)}
  \end{equation*}
\end{definition}
With the augmented Chern character and Todd class defined, the main result 
we will use is Corollary 4.13 of \cite{Toen_1998}:
\begin{theorem}
    Assume $\mathcal{Y}$ is a tame, proper, and smooth Deligne-Mumford stack 
    such that the coarse moduli space is a quasiprojective scheme. Let $F$ be 
    a coherent sheaf. Then 
    \begin{equation*}
        \chi(\mathcal{Y}, F) = \int_{I\mathcal{Y}} \Ch(F) \Td(\mathcal{Y})
    \end{equation*}
    where the terms above are the augmented Chern and Todd classes 
    respectively.
\end{theorem}
In what follows define 
$$\ext^i(F,G) = \dim(\Ext^i(F,G))$$ 
and $\chi(F,G) = \sum_{i = 0}^{\infty} (-1)^i \ext^i(F,G)$. Note that with the 
same assumptions on $\eey$ as above necessarily its projective dimension is 
finite and in fact equal to $\dim(\eey)$. As such $\ext^i(F,G) = 0$
for all $i > \dim(\eey)$ and thus the above formula is always well defined.
Next, by passing from the sheaves $F,G$ to locally free resolutions in the 
bounded derived category 
$\textbf{F}^{\bullet}, \textbf{G}^{\bullet} \in D^b(\eey)$ we may argue as in 
Lemma 6.1.3 of \cite{Huybrechts_Lehn_2010} to prove the useful consequence
\begin{coro}\label{TwoSheafGRR}
    With the same assumptions on $\eey$, if $F,G \in \Coh(\eey)$ then
    \begin{equation*}
        \chi(F,G) = \int_{I \eey} \Ch(G) \Ch(F^{\vee}) \Td(\eey)
    \end{equation*}
\end{coro}
Fix a choice of $\zz \in \eex^{[\rho]}$. 
As the action is linear it extends to a linear action on $\p^n$ and hence
there is a projective stacky compactification of $\eex$. Moreso as $G$ is 
finite and the field of definition is $\C$ it follows the compactification is 
also a tame DM stack. As $\Ocal_{\zz}$ is a compactly supported sheaf it 
follows both of the above results above may be used to calculate 
$\chi(\Ocal_{\zz}, \Ocal_{\zz})$ by passing to the compactification and 
restricting via functoriality on some open substack. With this in mind, 
the remainder of this section consist of an analysis of the equality 
\begin{equation*}
    \chi(\Ocal_{\zz}, \Ocal_{\zz}) = \int_{I \eex} \Ch(\Ocal_{\zz}^{\vee}) \Ch(\Ocal_{\zz}) \Td(\eex)
\end{equation*}

\subsection{The Euler Form of Abelian Planar Stacks}
Recall the natural equivalence of categories 
$\Coh(\eex) \simeq \fgMod_{\C[x,y]}^G$ and fix a character isomorphism 
$\chi: G \simeq \Irr(G)$. Let $\Ocal(g) \in \Coh(\eex)$ denote the sheaf
which corresponds to $\C[x,y] \otimes \chi(g)$ where $\C[x,y]$ is given the 
trivial action of $G$. For $F \in \Coh(\eex)$ define 
$F(g) := F \otimes \chi(g)$. Next fix a substack $\zz \subset \eex$, set 
$\rho = \varrho(\zz)$, and form the isotypic decomposition 
$\rho \simeq \sum_{g \in G} a_g \chi(g)$.
By the additivity of Euler characteristics of coherent sheaves we may pass to 
$K(\eex)$. Upon doing so there is an equality of classes 
\begin{equation}\label{SubstackExpand}
  [\Ocal_{\zz}] = \sum_{g} a_g[\Ocal_0(g)]
\end{equation}
where $\Ocal_0$ corresponds to $\C \in \fgMod_{\C[x,y]}^G$ and is endowed 
with the trivial action. Already one observation may be made on 
$\chi(\Ocal_{\zz}, \Ocal_{\zz})$. 
\begin{prop}\label{EulerFormExist}
  $\chi(\Ocal_{\zz}, \Ocal_{\zz})$ is an integral quadratic form in the 
  coefficients $\{a_g\}$. In particular it depends only on 
  $\rho$.
  \begin{proof}
    The key point is that $\chi(\Ocal_{\zz}, \Ocal_{\zz})$ is bilinear 
    in its components. Upon using equation (\ref{SubstackExpand})
    it can be calculated 
    \begin{equation*}
      \chi(\Ocal_{\zz}, \Ocal_{\zz}) = \sum_{g,h} a_g a_h \chi(\Ocal_0(g), \Ocal_0(h))
    \end{equation*}
    so the simplified Euler characteristics are the coefficients of 
    the form. Integrality follows from the equality
    $\chi(F,G) = \sum_{i=0}^{\infty} (-1)^i \ext^i(F,G)$
    and that each ext term in the sum is an integer.
  \end{proof}
\end{prop}
As a consequence of the proposition we can naturally associate to $\eex$ a 
quadratic form:
\begin{definition}
  Let $\eex$ be an abelian planar stack. The Euler form of $\eex$ is the 
  quadratic form $Q_{\eex}(\{a_g\})$ with the property that if 
  $\zz \in \eex^{[\rho]}$ and $\rho = \sum_{g \in G} a_g \chi(g)$ 
  then $Q_{\eex}(\{a_g\}) = \chi(\Ocal_{\zz}, \Ocal_{\zz})$.
\end{definition}
Often $Q_{\eex}(\rho)$ will be used in lieu of $Q_{\eex}(\{a_g\})$.
Of course as defined $Q_{\eex}$ may be extended to also accept any 
coherent sheaf with zero dimensional support. One point we will return to is 
that the property $Q_{\eex}(\{a_g\}) = \chi(\Ocal_{\zz}, \Ocal_{\zz})$ does 
not uniquely define $Q_{\eex}$ unless $\eex$ is rigid.

\subsection{Tangent Space Dimensions}
 
Recall with the given action of $G$ we have defined $g_x,g_y \in G$
by the isomorphism of representations
$\Span(x) \simeq \chi(g_x), \Span(y) \simeq \chi(g_y)$ for $x,y \in \C[x,y]$.
This subsection will consist of calculating the terms of the following sum
\begin{equation*}
  \chi(\Ocal_{\zz}, \Ocal_{\zz}) = \sum_{i = 0}^{\infty} (-1)^i \ext^i(\Ocal_{\zz}, \Ocal_{\zz})
\end{equation*}
To begin, as $\eex$ is a smooth DM stack in turn it may be verified its 
projective dimension is equal to $\dim(\eex) = 2$. In particular,
$\ext^i(\Ocal_{\zz}, \Ocal_{\zz}) = 0$ when $i > 2$ and thus
\begin{equation*}
  Q_{\eex}(\rho) = \ext^0(\Ocal_{\zz},\Ocal_{\zz}) - \ext^1(\Ocal_{\zz},\Ocal_{\zz}) + \ext^2(\Ocal_{\zz},\Ocal_{\zz})
\end{equation*}
We will now calculate the remaining terms.
For $i = 0$ there is an isomorphism 
$\Ext^0(\Ocal_{\zz}, \Ocal_{\zz}) \simeq H^0(\zz, \Ocal_{\zz})$
so $\ext^0(\Ocal_{\zz}, \Ocal_{\zz}) = h^0(\Ocal_{\zz})$. By comparison 
it follows $\ext^0(\Ocal_{\zz}, \Ocal_{\zz}) = a_0$.
For $i = 1$, we have:
\begin{lemma}
    $\ext^1(\Ocal_{\zz}, \Ocal_{\zz}) = \dim(T_{\zz} \eex^{[\rho]})$
    \begin{proof}
        We will sketch the proof. Let $I_{\zz} \subset \Ocal_{\eex}$ denote 
        the ideal sheaf of $\zz$ and form the exact sequence
        \begin{equation*}
          0 \rightarrow I_{\zz} \rightarrow \Ocal_{\eex} \rightarrow \Ocal_{\zz} \rightarrow 0
        \end{equation*}
        By the induced long exact sequence of Ext we have 
        $\Ext^1(\Ocal_{\zz}, \Ocal_{\zz}) \simeq \text{Hom}_{\Ocal_{\eex}}(I_{\zz}, \Ocal_{\zz})$ where we have used Ext vanishes on locally free 
        sheaves and 
        $$\text{Hom}_{\Ocal_{\eex}}(\Ocal_{\zz}, \Ocal_{\zz}) \simeq \text{Hom}_{\Ocal_{\eex}}(\Ocal_{\eex}, \Ocal_{\zz})$$ 
        From here the remainder of the proof follows from the isomorphism 
        $T_{\zz} \eex^{[\rho]} \simeq \text{Hom}_{\Ocal_{\eex}}(\Ocal_{\eex}, \Ocal_{\zz})$ which is proven by recreating the 
        argument of Section 2.2 of \cite{Huybrechts_Lehn_2010}.
    \end{proof}
\end{lemma}
For $i = 2$, let $R = \C[x,y]$ and note that the cotangent bundle of $\eex$ 
corresponds to the module $R[dx,dy]$. As we have defined the elements 
$g_x,g_y$ by the isomorphisms of representations $\Span(x) \simeq \chi(g_x)$,
$\Span(y) \simeq \chi(g_y)$ by rules of differentials we have 
$d(g \cdot x) = \chi_{g_x}(g) dx, d(g \cdot y) = \chi_{g_y}(g) dy$.
From this it is clear the canonical bundle satisfies $\omega \simeq \Ocal(c)$
where $c = g_x + g_y$. Due to this $c$ will be referred to as the 
\emph{canonical element} of the action. 
\begin{lemma}
  $\ext^2(\Ocal_{\zz}, \Ocal_{\zz}) = a_{-c}$
    \begin{proof}
        We will sketch the proof. There is a Serre Duality for derived 
        categories, and moreso it is compatible with the group action. Thus 
        we will calculate
        \begin{equation*}
          \begin{split}
            \Ext^2(\Ocal_{\zz}, \Ocal_{\zz}) 
            &= \Ext^0(\Ocal_{\zz}, \Ocal_{\zz} \otimes \omega)^{\vee} = \text{Hom}_{\Ocal_{\eex}}(\Ocal_{\zz}, \Ocal_{\zz}(c))^{\vee}\\
          \end{split}
        \end{equation*}
        where the last equality follows from 
        $\text{Ext}^0 = \text{Hom}_{\Ocal_{\eex}}$ and the isomorphism
        $\omega_{\eex} \simeq \Ocal(c)$. It may be verified there is 
        an isomorphism $\text{Hom}_{\Ocal_{\eex}}(\Ocal_{\zz}, \Ocal_{\zz}(c)) \simeq H^0(\Ocal_{\zz}(c))$.
        Ultimately we only care about the dimensions of these vector spaces 
        and thus we will avoid dualizing. Then global sections of any coherent 
        sheaf on $\eex$ are in a bijective correspondence with the invariant 
        global sections of the lift to $\aff^2$. As 
        $\omega \simeq \Ocal(c)$, upon comparison with the definition 
        of $a_{-c}$ it follows 
        \begin{equation*}
            \ext^2(\Ocal_{\zz}, \Ocal_{\zz}) = \dim\big((\rho \otimes \chi(c))^G\big) = a_{-c}
        \end{equation*}
    \end{proof}
\end{lemma}
These terms may then be substituted into the alternating 
sum expression for $Q_{\eex}$. Upon doing so we may rearrange the terms 
to prove:
\begin{equation*}
  t_{\eex}(\rho) = \dim(T_{\zz} \eex^{[\rho]}) = a_0 + a_{-c} - Q_{\eex}(\rho)
\end{equation*}
Note that this implies the dimension of tangent spaces for abelian planar 
stacks is necessarily constant as the right hand side depends 
only on $\rho$; This gives a stack theoretic interpretation of the 
smoothness of planar Multigraded Hilbert Schemes as in Theorem 1.1 of 
\cite{Maclagan_Smith_2010}. Then by smoothness with either proof 
we have the following dimension formula:
\begin{prop}\label{ConsTanDim}
  For all $\rho \in \Rep(G)$ such that $\eex^{[\rho]} \neq \emptyset$ we have 
  \begin{equation*}
    \dim(\eex^{[\rho]}) = a_0 + a_{-c} - Q_{\eex}(\rho)
  \end{equation*}
\end{prop}

\subsection{Euler Forms of Planar DM Stacks}

Although of theoretical interest, it is preferable to be able to give 
a more explicit description of Euler forms. The first observation is 
\begin{lemma}
  If $\eex$ is rigid and $\eex^{[\rho]} \neq \emptyset$ then
  $t_{\eex}(\rho + \RR) = t_{\eex}(\rho) + 2$
  \begin{proof} 
    By assumption there exists $\zz \in \eex^{[\rho]}$ and as $\eex$ is 
    rigid in addition there exists a simple geometric 
    point $p \in \Delta$ not contained in $\zz$. Let $\xi$ be the subscheme 
    corresponding to $\zz$ and $f$ the subscheme corresponding to $p$ 
    and note from the set up both are disjoint subschemes. Now as $p$ 
    is a nonstacky point in turn $\varrho(f) \simeq \RR$ by Proposition 
    \ref{DeltaRepCase}, and by disjointness we have 
    $\varrho(\xi \sqcup f) \simeq \varrho(\xi) + \varrho(f) = \rho + \RR$.
    Thus there is the equality 
    $t_{\eex}(\rho + \RR) = \dim(T_{\zz \sqcup p} \eex^{[\rho + \RR]})$.
    As shown above there is an isomorphism of groups 
    \begin{equation*}
      T_{\zz \sqcup p} \eex^{[\rho + \RR]} \simeq \text{Hom}^G(I(\xi \sqcup f), H^0(\xi \sqcup f))
    \end{equation*}
    where the Hom is taken as $\C[x,y]$-modules. By disjointness of the 
    subschemes in turn
    \begin{equation*}
      \text{Hom}^G(I(\xi \sqcup f), H^0(\xi \sqcup f))
      \simeq \text{Hom}^G(I(\xi), H^0(\xi)) \oplus \text{Hom}^G(I(f), H^0(f))
    \end{equation*}
    By another application of the tangent vector characterization we can 
    see
    \begin{equation*}
        T_{\zz \sqcup p} \eex^{[\rho + \RR]} \simeq T_{\zz} \eex^{[\rho]} 
        \oplus T_{p} \eex^{[\RR]}
    \end{equation*}
    so that by taking dimensions we obtain the equality 
    $t_{\eex}(\rho + \RR) = t_{\eex}(\rho) + \dim(T_{p} \eex^{[\RR]})$.
    To handle the remaining term we will use that $\Delta$ is open in 
    $\eex$ to the dimension may be calculated with the tangent space of 
    $\Delta$ at $p$. With the comments following Proposition 
    \ref{DeltaRepCase} it was also shown
    $\Delta^{[\RR]} \simeq \delta$
    for an open subscheme of $\aff^2/G$ consisting of smooth points. In total 
    this implies the last tangent space is two dimensional and thus
    $t_{\eex}(\rho + \RR) = t_{\eex}(\rho) + 2$.
  \end{proof}
\end{lemma}
Now substitute $\rho + \RR$ into Proposition \ref{ConsTanDim} to show
\begin{equation*}
  \begin{split}
    t_{\eex}(\rho + \RR) 
    &= (a_0 + 1) + (a_{-c} + 1) - Q_{\eex}(\rho + \RR)\\ 
    &= a_0 + a_{-c} - Q_{\eex}(\rho + \RR) + 2 \\
  \end{split}
\end{equation*}
On the other hand, the lemma implies
\begin{equation*}
  \begin{split}
    t_{\eex}(\rho + \RR) 
    &= t_{\eex}(\rho) + 2 = a_0 + a_{-c} - Q_{\eex}(\rho) + 2 \\
  \end{split}
\end{equation*}
so that by comparison we have
\begin{prop}\label{QCons}
  If $\eex^{[\rho]} \neq \emptyset$ then 
  $Q_{\eex}(\rho + \RR) = Q_{\eex}(\rho)$
\end{prop}
Next note $0 \in \eex^{[\chi(0)]}$ so that 
$Q_{\eex}(\chi(0) + \RR) = Q_{\eex}(\chi(0))$. On its own this is not 
interesting however we will observe the following property of 
Euler characteristics.
\begin{lemma}\label{OrthoRelation}
  Let $\eex = [\aff^2/G]$. Then for all $[F_1],[F_2] \in K(\eex)$ and 
  $g \in G$ we have 
  \begin{equation*}
    \chi([F_1(g)], [F_2(g)]) = \chi([F_1],[F_2])
  \end{equation*}
  \begin{proof}
    We know that 
    $\chi([F_1],[F_2]) = \sum_{i=0}^{\infty} (-1)^i \ext^i(F_1,F_2)$, and 
    this sum is always finite. Then by properties of Ext functors we have 
    \begin{equation*}
      \text{Ext}^i(F_1(g),F_2(g)) \simeq \text{Ext}^i(F_1,F_2(g) \otimes \Ocal(-g)) \simeq \text{Ext}^i(F_1,F_2)
    \end{equation*}
    for all $i$.
  \end{proof}
\end{lemma}
Now consider the vector space $\R[G]$ endowed with the canonical basis 
$\{e_g\}_{g \in G}$ and let $\pi_g$ denote the projection to the $g$-th 
coordinate. In this way the Euler characteristic induces a bilinear form on 
$\R[G]$ by setting 
\begin{equation*}
  \innprod{v}{w}_{\eex} = \sum_{g,h \in G} \pi_g(v) \pi_g(w) \chi([\Ocal(g)], [\Ocal(h)])
\end{equation*}
Now we may prove an orthogonality relation for Euler forms:
\begin{prop}
  $Q_{\eex}(g \cdot v) = Q_{\eex}(v)$ for all $g \in G$ and $v \in \R[G]$.
  \begin{proof}
    By the previous proposition we have 
    $\innprod{g \cdot v}{g \cdot w}_{\eex} = \innprod{v}{w}_{\eex}$
    and hence the left permutation representation on $\R[G]$ acts by orthogonal
    matrices. By similarly extending $Q_{\eex}$ to $\R[G]$ it is clear the 
    Euler form is the quadratic form associated to $\innprod{-}{-}_{\eex}$. 
    Hence by the above orthogonality of the left regular representation we 
    have $Q_{\eex}(g \cdot v) = Q_{\eex}(v)$ for all $v \in \R[G]$. 
  \end{proof}
\end{prop}
\begin{remark}
  In fact since the proof of the lemma does not depend on $\eex$ 
  being planar, this shows Euler forms of stacks of the 
  form $[\aff^n/G]$ with $G$ abelian and $n$ arbitrary also satisfy
  $Q_{\eex}(g \cdot \rho) = Q_{\eex}(\rho)$. Moreso there is an 
  appropriate generalization for when $G$ is nonabelian if we further 
  restrict to $g \in Z(G)$.
\end{remark}
We will prove one last functional equation of $Q_{\eex}$ which extends 
Proposition \ref{QCons} to all of $\R[G]$; Note that unlike the orthogonality 
property, this does use constancy of dimension in an essential way and hence 
is currently unique to planar stacks.
\begin{lemma}
  Set $w_G = \sum_{g \in G} e_g$. Then 
  $Q_{\eex}(v + w_G) = Q_{\eex}(v)$ for all $v \in \R[G]$.
  \begin{proof}
    By Lemma \ref{QCons} we have 
    $Q_{\eex}(\chi(0) + \RR) = Q_{\eex}(\chi(0))$. In terms of vectors of 
    $\R[G]$ the previous equality implies 
    $Q_{\eex}(e_0 + w_G) = Q_{\eex}(e_0)$ since $e_0$ corresponds 
    to the sheaf $[\Ocal(0)]$.  On the other hand 
    observe $g \cdot e_0 = e_g$ and $g \cdot (e_0 + w_G) = e_g + w_G$
    since $v_G$ is a simultaneous eigenvector of the action with an 
    eigenvalue of 1. Thus $Q_{\eex}(g \cdot v) = Q_{\eex}(v)$
    implies $Q_{\eex}(e_g + w_G) = Q_{\eex}(e_g)$. On induction this 
    implies $Q_{\eex}(e_g + n \cdot w_G) = Q_{\eex}(e_g)$ for all $n \in \N$.

    Next let $A$ denote the symmetric matrix associated to $Q_{\eex}$.
    By direct calculation we find for any choice of $v$ that 
    $$Q_{\eex}(v + w_G) = (v + w_G)^T A (v + w_G) = v^T A v + 2v^T A w_G + w_G^T A w_G$$
    Note $\ell(v) = 2 v^T A w_G$ is linear in $v$. Next substituting 
    $v = e_g + n \cdot w_G$ yields
    $$Q_{\eex}(e_g + n \cdot w_G) = e_g^T A e_g + 2ne_g^T A w_G + n^2 (w_G^T A w_G)$$
    As $Q_{\eex}(e_g + n \cdot w_G) = Q_{\eex}(e_g) = e_g^T A e_g$ this 
    implies 
    $2n e_g^T A w_G + n^2 (w_G^T A w_G) = 0$ for each $n \in \N$. As the 
    equation holds for $n = 1, 2$ in turn this implies 
    \begin{equation*}
      \begin{split}
        0 
        &= (4 e_g^T A w_G + 4 w_G^T A w_G) - 2 \cdot (2 e_g^T A w_G + w_G^T A w_G)\\
        &= 2 w_G^T A w_G\\
      \end{split}
    \end{equation*}
    and in turn $\ell(e_g) = 2e_g^T A w_G = 0$ for each $g \in G$. As 
    $\ell$ is linear and zero on a basis in turn $\ell(v) = 0$ for all 
    $v$. Finally we have already shown $w_G^T A w_G = 0$ so that 
    upon substituting these into the above equations we have 
    $$Q_{\eex}(v + w_G) = v^T A v = Q_{\eex}(v)$$
  \end{proof}
\end{lemma}
As $Q_{\eex}$ is a quadratic form this is a strong condition to possess;
The following result is due to Sudheer Shukla.
\begin{lemma}\label{SudheersLemma}
    Let $V = \R^d$ be a vector space whose canonical basis is $\{e_i\}$,
    and set $\overline{1} = (1,...,1)$. Let $v = (a_1,...,a_{d}) \in V$ and
    let $Q$ be an integral quadratic form such that 
    $Q(v + \overline{1}) = Q(v)$ for all $v$. Then there exists 
    $b_{i,j} \in \frac{1}{2} \Z$ such that
    \begin{equation*}
        Q(a_1,..., a_n) = \sum_{i < j} b_{i,j}(a_i - a_j)^2
    \end{equation*}
    \begin{proof}
        Let $A$ denote the matrix associated to $Q$ with respect to 
        $\{e_i\}$. The condition of the hypothesis amounts to the equality
        \begin{equation*}
            v^T A ~v = (v + \overline{1})^T A ~ (v + \overline{1}) = v^T A ~ v + 2 v^T A  ~ \overline{1}+ \overline{1}^T A ~ \overline{1}
        \end{equation*}
        from which we can deduce 
        $2 v^T A ~ \overline{1} + \overline{1}^T A ~\overline{1} = 0$.
        As this holds for all $v$ necessarily 
        $v^T A ~ \overline{1}, \overline{1}^T A ~ \overline{1} = 0$.
        By sampling $e_i$ in the first expression we can 
        see that each row sum of $A$ must be equal to 0, from which the 
        second equality is implied as $A \cdot \overline{1}$ is equal to the 
        vector of row sums. In fact the opposite direction holds as 
        well so the the hypothesis is satisfied iff each row sum of $A$ is 
        equal to 0.

        Next write $A = (b_{i,j})$ and observe 
        \begin{equation*}
        \begin{split}
            \sum_{i < j} -b_{i,j}(a_i - a_j)^2 
            &= \sum_{i = 1}^n \bigg( \sum_{j \neq i} -b_{i,j} a_i^2 + \sum_{i \neq j} 2 b_{i,j} a_i a_j\bigg)\\
            &= \sum_{i = 1}^n \bigg( b_{i,i} a_i^2 + \sum_{i \neq j} 2 b_{i,j} a_i a_j\bigg) = Q(a_1,...,a_n)\\
        \end{split}
        \end{equation*}
        where we have used $\sum_{i = 1}^n b_{i,j} = 0$ in the last equality.
        To see the claim on the coefficients recall $Q_{\eex}$ is an 
        integral quadratic form so that $b_{i,j} \in \frac{1}{2}\Z$.
    \end{proof}
\end{lemma}
Putting everything together we have proven
\begin{prop}\label{CorpulentEulerForm}
Suppose $\eex$ is rigid, and let $D = (G^2 - \Delta)/S_2$ denote the 
set of unordered pairs of distinct elements of $G$. Then there exists 
$b_{g,h} \in \frac{1}{2} \Z$ such that if $\rho \in \Rep(G)$ and
$\rho \simeq \sum_{g \in G} a_g \chi(g)$ then
\begin{equation*}
  Q_{\eex}(\rho) = \sum_{(g,h) \in D} b_{g,h} (a_g - a_h)^2
\end{equation*}
\end{prop}

\section{Calculation of Euler Forms}

Recall that we have shown if $\rho \in \Rep(G)$ and if $H$ is the stabilizer 
group of $\aff^2$ then $\eex^{[\rho]} \neq \emptyset$ implies 
$\rho \in \Rep(G/H)$ with respect to the canonical inclusion 
$\Rep(G/H) \subset \Rep(G)$, and 
$\eex^{[\rho]} \simeq \eexgr^{[\rho]}$ for all $\rho \in \Rep(G/H)$.
Moreso we have also calculated the dimension formula 
\begin{equation*}
  \dim(\eex^{[\rho]}) = a_0 + a_{-c} - Q_{\eex}(\rho)
\end{equation*}
Thus to obtain complete information on $\dim(\eex^{[\rho]})$ it 
suffices to calculate $Q_{\eex}$. As such this section will consist of the 
providing an explicit equation for the coefficients of the Euler form.

\subsection{The GRR Integral Calculation}

For the moment assume $G$ is a finite abelian group acting linearly on 
$\aff^2$ and let $\eex = [\aff^2/G]$ be the corresponding quotient stack.
We will begin by setting up the calculation of 
the GRR integral of \cite{Toen_1998}, but first we will make some remarks on 
notation and conventions for the calculation.

$\bullet$ 
In what follows we will make implicit use of the natural equivalence 
of categories $\Coh(\eex) \simeq \fgMod_{\C[x,y]}^G$ to define sequences in 
terms of $G$-modules of $\C[x,y]$ and pass to associated sequences of sheaves 
on $\eex$. $K(\eex)$ denotes the Grothendieck group of vector bundles. 
$[F]^{\vee}$ refers to the ``dualizing" operation on $K(\eex)$ which is 
defined on vector bundles by $[V]^{\vee} = [V^{\vee}]$ and defined on 
the whole of $K(\eex)$ by extending linearly. Define 
$\II := \Ch([\Ocal_{\zz}]^{\vee}) \Ch([\Ocal_{\zz}]) \Td(\eex)$.

$\bullet$ 
Fix a choice of character isomorphism $\chi: G \simeq \Irr(G)$. Recall 
$g_x \in G$ is defined by $\Span(x) \simeq \chi(g_x)$, and similarly 
for $g_y \in G$. $c = g_x + g_y$ is the canonical element of the action, and
$[L_x], [L_y] \in K$ denote the classes of the $x$ and $y$-axis endowed 
with the tangent action so $[L_x] = [\Ocal(-g_x)], [L_y] = [\Ocal(-g_y)]$. 

Now we will return to the calculation:
Let $\rho \in \Rep(G)$. In $K(\eex)$ for each $\zz \in \eex^{[\rho]}$ there is 
an equality 
\begin{equation*}
  [\Ocal_{\zz}] = \sum_{g \in G} a_g [\Ocal_0(g)]
\end{equation*} 
For $R = \C[x,y]$ recall the Koszul resolution 
\begin{equation*}
    0 \rightarrow \bigwedge^2 R^{2} \rightarrow \bigwedge^1 R^{2} \rightarrow R \rightarrow R/(x,y) \rightarrow 0
\end{equation*}
As the sequence is $G$-equivariant it gives rise to a resolution of sheaves 
on $\eex$,
\begin{equation*}
    0 \rightarrow L_x \otimes L_y \rightarrow L_x \oplus L_y \rightarrow \Ocal_{\eex} \rightarrow \Ocal_0 \rightarrow 0
\end{equation*}
Thus upon passing to $K$ we obtain
\begin{equation}\label{SimplePtClass}
    [\Ocal_0] = [\Ocal_{\eex}] - [L_x] - [L_y] + [L_x]\cdot[L_y]
\end{equation}
In order to proceed, we will take a moment to discuss the inertia stack.
Recall $I \eex = \bigsqcup_{g \in G} \CC_g$ where 
$\CC_g = [\Fix(\aff^2,g)/G]$. Via the disjoint union above there is a 
decomposition
\begin{equation*}
    A^*(I\eex) = \bigoplus_{g \in G} A^*(\CC_g)
\end{equation*}
Use $\pi_g: A^*(I \eex) \rightarrow A^*(\CC_i)$ to denote the projection maps
and $f_g = \pi_g(f)$ for $f \in A^*(I \eex)$. At this point a simplification 
to the analysis may be performed by reducing to zero dimensional components 
of $I\eex$. Assume $\dim(\CC_g) =  1$ and let 
$X = c_1(\pi_g^*([L_x])), Y = c_1(\pi^*([L_y]))$. Formally there is the 
equality
\begin{equation*}
    1 - e^X - e^Y+ e^{X+Y} = XY + \text{ higher order terms}
\end{equation*}
Then by equation \ref{SimplePtClass} and using that augmented Chern characters 
are ring homomorphisms it follows 
As such $\Ch([\Ocal_{\zz}]), \Ch([\Ocal_{\zz}]^{\vee})$ are at least quadratic 
in $X,Y$. As $\dim(\CC_g) = 1$ this implies the contribution of 
$\int_{\CC_g} \II_g = 0$ and the same holds true if $\dim(\CC_g) = 2$. 
As such we can see the only non-negligible contributions to the integral 
are from those $\CC_g$ such that $\dim(\CC_g) = 0$. 

The next step is to use $[\Ocal_{\zz}] = \sum_g a_g [\Ocal_0(g)]$ to give a 
similar expansion for $[\Ocal_{\zz}^{\vee}]$ with the aid of the following 
lemma.
\begin{lemma}
    For all $g \in G$ we have $[\Ocal_0(g)]^{\vee} = [\Ocal_0(c - g)]$.
    \begin{proof}
        Consider the Koszul resolution 
        \begin{equation*}
            0 \rightarrow L_x \otimes L_y \rightarrow L_x \oplus L_y \rightarrow \Ocal_{\eex} \rightarrow \Ocal_0 \rightarrow 0
        \end{equation*}
        which will be regarded as a complex. As all terms to the left of 
        $\Ocal_0$ are free modules, to ensure exactness 
        of the dual sequence it suffices to check exactness only at 
        $\Ocal_{\eex}$. However exactness at this location is met since 
        $\Ocal_{\eex}$ is an invertible sheaf, the dual map is 
        nonzero, and $\Ocal_{0}^{\vee} = 0$. Following this, explicitly write 
        out the dual sequence and tensor by 
        $L_x \otimes L_y \simeq \Ocal(-c)$ to obtain the exact sequence
        \begin{equation*}
            0 \rightarrow \Ocal_0^{\vee}(-c) \rightarrow L_x \otimes L_y \rightarrow L_y \oplus L_x \rightarrow \Ocal_{\eex} \rightarrow  0
        \end{equation*}
        As this is still exact, passing to the $K$ shows 
        \begin{equation*}
            [\Ocal_0]^{\vee} \cdot [\Ocal(c)]= [L_x] \cdot [L_y] - [L_x]-[L_y] + [\Ocal_{\eex}]
        \end{equation*}
        By self-duality of the Koszul complex the right hand side is equal to 
        $[\Ocal_0]$. Hence by rearranging we obtain 
        $[\Ocal_0]^{\vee} \simeq [\Ocal_0(c)]$. For the full claim, note 
        \begin{equation*}
        \begin{split}
            [\Ocal_0(g)]^{\vee} 
            &= [\Ocal_0]^{\vee} \cdot [\Ocal(g)]^{\vee} 
            = [\Ocal_0(c)] \cdot [\Ocal(-g)] = [\Ocal_0(c - g)]
        \end{split}
        \end{equation*}
    \end{proof}
\end{lemma}
Applying the lemma to each term of $[\Ocal_{\zz}]$ in the above sum proves:
\begin{equation*}
    \begin{split}
        [\Ocal_{\zz}]^{\vee} 
        &= \sum_{g \in G} a_{g} [\Ocal_0(c-g)]\\
    \end{split}
\end{equation*}

\subsection{The Augmented Chern Character and Todd Class}

We will now calculate the individual factors of $\II$. 
As we are working within $\widetilde{A}^*(I\eex)$ necessarily only those 
components $\CC_g$ such that $\dim(\CC_g) = 0$ have a nonzero contribution. The 
first calculation will be that of the augmented Chern character.
\begin{lemma}
  Assume $\dim(\CC_g) = 0$. Then 
  $\Ch([L_x])_g = \chi_{g}(-g_x), \Ch([L_y])_g = \chi_{g}(-g_y)$.
  \begin{proof}
    For the canonical projection, $q: I \eex \rightarrow \eex$, 
    there are isomorphisms obtained from pulling back,
    $q^* [L_x]_g \simeq [\Ocal(-g_x)]_g$. Next 
    $\Ch = \text{Ch} \circ \rho$, where $\rho$ is the eigenvalue map. By 
    comparison the element $g$ acts on $[\Ocal(-g_x)]$ with an eigenvalue of 
    $\chi_g(-g_x)$. As such 
    \begin{equation*}
      \begin{split}
        \Ch([L_x])_g 
        &= \text{Ch}\big(\rho([\Ocal(-g_x)]_g)\big) \\
        &= \text{Ch}(\chi_g(-g_x) \cdot [\Ocal(-g_x)]_g]) \\
        &= \chi_g(-g_x)\text{Ch}([\Ocal(-g_x)]_g]) = \chi_{g}(-g_x)\\
      \end{split}
    \end{equation*}
    The penultimate equality arises from pulling $\chi_g(-g_x)$ out of 
    $\text{Ch}$ as it is a complex root of unity, and the last equality 
    uses that $c_1([\Ocal(-g_x)] \vert_{\CC_g}) = 0$ as $\dim(\CC_g) = 0$. 
    The calculation for $L_y$ is identical and will be omitted.
  \end{proof}
\end{lemma}
Next note we may factor equation \ref{SimplePtClass} as
\begin{equation*}
  \begin{split}
      [\Ocal_0] 
      &= [\Ocal_{\eex}] - [L_x] - [L_y] + [L_x] \cdot [L_y]\\
      &= ([\Ocal_{\eex}] - [L_x])([\Ocal_{\eex}] - [L_y])\\
  \end{split}
\end{equation*}
within $K(\eex)$. As $q: I \eex \rightarrow \eex$ induces a 
ring homomorphism $q^*: A^*(\eex) \rightarrow A^*(I \eex)$
and the ring structure of the latter Chow ring is compatible with 
the decomposition $A^*(I \eex) \simeq \bigoplus_{g \in G} A^*(\CC_g)$ 
it follows that 
\begin{equation*}
  \begin{split}
      q^*[\Ocal_0]_g 
      &= ([\Ocal_{\CC_g}] - q^*[L_x]_g)([\Ocal_{\CC_g}] - q^*[L_y]_g)\\
  \end{split}
\end{equation*}
for all $g \in G$. Thus by using the calculations of the previous lemma 
and the fact that $\Ch$ is a ring homomorphism we obtain:
\begin{coro}
  If $\dim(\CC_g) = 0$ then 
  \begin{equation*}
    \Ch([\Ocal_0])_g = (1 - \chi_g(-g_x))(1 - \chi_g(-g_y))
  \end{equation*}
\end{coro}
Next observe 
$$\Ch([\Ocal_0(h)])_g = \Ch([\Ocal_0])_g \Ch([\Ocal(h)])_g = \chi_g(h) \Ch([\Ocal_0])$$
As $[\Ocal_{\zz}] = \sum_{g \in G} a_g [\Ocal_0(g)]$ we may multiply 
by the appropriate coefficient and sum the resulting terms together to obtain 
an explicit formula for $\Ch([\Ocal_{\zz}])$. Upon applying the same steps 
to $\Ch([\Ocal_{\zz}]^{\vee})$ and using that $[\Ocal_{\zz}(h)]^{\vee} = [\Ocal_0(c - h)]$ we may prove
\begin{prop}
    \begin{equation*}
        \Ch([\Ocal_{\zz}])_g = (1-\chi_g(-g_x))(1-\chi_g(-g_y)) A_g
    \end{equation*}
    \begin{equation*}
        \Ch([\Ocal_{\zz}^{\vee}])_g = (1-\chi_g(-g_x))(1-\chi_g(-g_y)) B_g
    \end{equation*}
    where
    \begin{equation*}
        A_g = \sum_{h \in G} a_h \chi_g(h), ~ B_i = \sum_{h \in G} a_{h} \chi_g(c - h)
    \end{equation*}
\end{prop}
The final remaining calculation is that of the augmented Todd class.
\begin{lemma}
    If $\dim(\CC_g) = 0$ then
    \begin{equation*}
        \Td(\eex)_g = (1 - \chi_g(-g_x))^{-1}(1 - \chi_g(-g_y))^{-1}
    \end{equation*}
    \begin{proof}
        With the induced action on the tangent bundle we have the 
        isomorphism $T \eex \simeq L_x^{\vee} \oplus L_{y}^{\vee}$.
        Recall augmented Todd classes are multiplicative on sums so that
        $\Td([\eex]) = \Td([L_x]^{\vee})\Td([L_y]^{\vee})$.
        To calculate the right hand side we will use the original definition:
        \begin{equation*}
            \Td(V) = \frac{\text{Td}((q^*V)^{inv})}{\text{Ch}\big(\rho \circ \lambda_{-1}(\big[(q^*V)^{mov}\big]^{\vee}) \big)}
        \end{equation*}
        As $\dim(\CC_g) = 0$ in particular $g \neq 0$ so that 
        $(q^* L_x)_g^{inv} = 0$ and thus the numerator is equal to 1. For the 
        denominator, clearly 
        $(q^*L_x)_g^{mov} = [q^* L_x]_g = [\Ocal(-g_x)]_g$ and hence 
        its dual pulls back to $[\Ocal(g_x)]_g$. 
        Dualizing once more as in the equation for $\Td$ and
        calculating the remainder of the denominator yields:
        \begin{equation*}
            \begin{split}
                \text{Ch}\big(\rho \circ \lambda_{-1}([\Ocal(-g_x)]_g)\big)
                &= \text{Ch}\big(\rho([\Ocal_{\CC_g}] - [\Ocal(-g_x)]_g\big)\\
                &= \text{Ch}(1 - \chi_g(-g_x) \cdot [\Ocal(-g_x)]_g)\\
                &= 1 - \chi_g(-g_x) \cdot \text{Ch}([\Ocal(-g_x)_g])= 1 - \chi_g(-g_x)\\
            \end{split}
        \end{equation*}
        where the last equality follows from $c_1([\Ocal(-g_x)]_g) = 0$. The 
        calculation for $[L_y]$ is identical and will be omitted.
        By multiplicativity of Todd classes the claim is proven.
    \end{proof}
\end{lemma}
With the preceding lemma completed, each factor of $\II$ has been 
treated and hence the integral may be calculated. 
Substituting each of the above expressions into the equation of 
$\II_g$ yields 
\begin{equation*}
  \II_g = A_g B_g (1 - \chi_g(-g_x)) (1 - \chi_g(-g_y))
\end{equation*}
where $A_g, B_g$ are defined as above.
\begin{prop}
  \begin{equation*}
  Q_{\eex} = \frac{1}{\# G} \sum_{\dim(\CC_g) = 0} A_g B_g (1 - \chi_g(-g_x)) (1 - \chi_g(-g_y))
\end{equation*}
\begin{proof}
  Note that by the above calculations if $\dim(\CC_g) = 0$ then 
  $\II_g$ is a constant term. Moreso we have already observed that 
  if $\dim(\CC_g) > 0$ then the corresponding integral on the component 
  is equal to 0. As such we will calculate
  \begin{equation*}
  \begin{split}
      \int_{I \eex} \II 
      &= \sum_{g \in G} \int_{\CC_g} \II_g 
      = \sum_{\dim(\CC_g) = 0} \int_{\CC_g} \II_g \\
      &= \frac{1}{\# G} \sum_{\dim(\CC_g) = 0} A_g B_g (1 - \chi_g(-g_x)) (1 - \chi_g(-g_y))
  \end{split}
  \end{equation*}
  The factor of $\frac{1}{\# G}$ arises from the fact that each such $\CC_g$ 
  is a zero dimensional gerbe with stabilizer group of size $\# G$.
\end{proof}
\end{prop}
Now consider the completed sum,
\begin{equation*}
  \overline{Q}_{\eex} = \frac{1}{\# G} \sum_{g \in G} A_g B_g (1 - \chi_g(-g_x)) (1 - \chi_g(-g_y))
\end{equation*}
where we now range over the entire group. Observe that if 
$\dim(\CC_g) > 0$ then necessarily either $q^* [L_x]_g$ or $q^* [L_y]_g$
is the trivial line bundle. By the above calculations this implies 
$\chi_g(-g_x) = 1$ or $\chi_g(-g_y) = 1$. As such for those 
$g \in G$ such that $\dim(\CC_g) > 0$ we have 
\begin{equation*}
  (1 - \chi_g(-g_x)) (1 - \chi_g(-g_y)) = 0
\end{equation*}
In total we obtain the equality of forms
\begin{equation*}
  \overline{Q}_{\eex} = Q_{\eex}
\end{equation*}
The remainder of this section will consist of calculating the coefficients 
of the Euler form by instead performing a suitable analysis of 
$\overline{Q}_{\eex}$.

\subsection{Simplification of the Euler Form}

To proceed, we know already there exists $b_{g,h} \in \frac{1}{2} \Z$ 
such that
\begin{equation*}
  Q_{\eex} = \sum_{(g,h) \in D} b_{g,h} (a_g - a_h)^2
\end{equation*}
where $D = (G^2 - \Delta)/S_2$. Furthemore it was already shown 
that Proposition \ref{OrthoRelation} implies the left regular 
representation of $G$ on $\R[G]$ acts by orthogonal matrices 
with respect to $\innprod{-}{-}_{\eex}$. In particular the 
symmetric matrix associated to $Q_{\eex}$ is circulant so 
$b_{g'+g,g'+h} = b_{g,h}$ for all $g',g,h \in G$. Thus without loss of 
generality set $h = 0$ and take $g \neq 0$. In what follows let 
$\partial_g = \frac{\partial}{\partial a_g}$ 
and $e_0$ denote the canonical basis element of $\R[G]$ corresponding to 
$0 \in G$. Observe that $\partial_g Q_{\eex}(e_0) = -2b_{g,0}$. 
To obtain more desirable equalities later we may just as well calculate 
the partial derivatives of $\overline{Q}_{\eex}$ in order to find the 
$b_{g,0}$. From its definition we have
\begin{equation*}
  \overline{Q}_{\eex} = \frac{1}{\# G} \sum_{g \in G} A_g B_g (1 - \chi_g(-g_x)) (1 - \chi_g(-g_y))
\end{equation*}
and 
\begin{equation*}
    A_g = \sum_{h \in G} a_h \chi_g(h), ~ B_i = \sum_{h \in G} a_{h} \chi_g(c - h)
\end{equation*}
Let $\theta \in G$ denote a generic element. Upon doing so we will 
expand the product $A_{\theta} B_{\theta}$ and regroup terms to obtain
\begin{equation*}
    \begin{split}
        A_{\theta} B_{\theta} 
        &= \sum_{h_1, h_2 \in G} a_{h_1} a_{h_2} \chi_{\theta}(h_1) \chi_{\theta}(c - h_2)\\
        &= \sum_{h_1 = h_2} a_{h_1}^2 \chi_{\theta}(h_1) + \sum_{(h_1,h_2) \in D} a_{h_1}a_{h_2} (\chi_{\theta}(c + h_1 - h_2) + \chi_{\theta}(c + h_2 - h_1))
    \end{split}
\end{equation*}
where we have used $\chi_{\theta}(a)\chi_{\theta}(b) = \chi_{\theta}(a + b)$.
Take a partial derivative with respect to $a_g$ to obtain 
\begin{equation*}
  \partial_g \big(A_{\theta} B_{\theta} \big) = 2a_g \chi_{\theta}(g) + \sum_{h_2 \neq g} a_{h_2} \big(\chi_{\theta}(c + g - h_2) + \chi_{\theta}(c + h_2 - g) \big)
\end{equation*}
whence upon substituting $e_0$ and using that $g \neq 0$ we obtain 
\begin{equation*}
    \partial_g \big(A_{\theta} B_{\theta} \big)(e_0) = \chi_{\theta}(c + g) + \chi_{\theta}(c - g)
\end{equation*}
As $c = g_x + g_y$, upon multiplying the above expression by 
\begin{equation*}
  (1 - \chi_{\theta}(-g_x))(1 - \chi_{\theta}(-g_y)) = 1 - \chi_{\theta}(-g_x) - \chi_{\theta}(-g_y) + \chi_{\theta}(-c)
\end{equation*}
it follows
\begin{equation*}
  \partial_g \big(A_{\theta} B_{\theta} \big)(e_0) \big(1 - \chi_{\theta}(-g_x) - \chi_{\theta}(-g_y) + \chi_{\theta}(-c)\big)
\end{equation*}
is equal to
\begin{equation*}
  \begin{split}
    \sum_{\ell \in \{0,g_x,g_y,c\}} \chi_{\theta}(\ell + g) + \chi_{\theta}(\ell - g)
  \end{split}
\end{equation*}
Observe that for any fixed $h \in G$ we have 
$\chi_{\theta}(h) = \chi_{h}(\theta)$. Thus upon summing over 
$\theta \in G$ and dividing by the order of $G$ we obtain
\begin{equation*}
  \begin{split}
    \partial_g \overline{Q}_{\eex}(e_0) 
    &= \frac{1}{\# G} \sum_{\theta \in G} \bigg( \sum_{\ell \in \{0,g_x,g_y,c\}} \chi_{\theta}(\ell + g) + \chi_{\theta}(\ell - g)\bigg)\\
    &= \frac{1}{\# G} \sum_{\theta \in G} \bigg( \sum_{\ell \in \{0,g_x,g_y,c\}} \chi_{\ell + g}(\theta) + \chi_{\ell - g}(\theta)\bigg)\\
    &= \frac{1}{\# G} \sum_{\ell \in \{0,g_x,g_y,c\}} \bigg( \sum_{\theta \in G} \chi_{\ell + g}(\theta) + \chi_{\ell - g}(\theta)\bigg)\\
  \end{split}
\end{equation*}
Next we will rewrite the summand as 
\begin{equation*}
  \chi_{\ell + g}(\theta) + \chi_{\ell - g}(\theta) = 
  \chi_{\ell}(\theta) \overline{\chi_{- g}(\theta)} + \chi_{\ell}(\theta) \overline{\chi_{g}(\theta)}
\end{equation*}
Recall that for any $G$ finite over $\C$ there is an inner product on 
characters given by 
\begin{equation*}
    \innprod{\rho_1}{\rho_2} = \frac{1}{\# G} \sum_{g \in G} \rho_1(g) \overline{\rho_2(g)}
\end{equation*}
By comparison with the previous expression we may deduce
the summand is equal to 
$\innprod{\chi_{\ell}}{\chi_{- g} + \chi_g}$.

Define the \emph{standard character} by
$\rhostd = \chi_{g_x} + \chi_{g_y}$. Note that 
$\innprod{\chi_0}{\chi_{g} + \chi_{-g}} = 0$ since by assumption 
$g \neq 0$. As such we may compactly write the above expression for the 
partial derivative as
\begin{equation*}
  \begin{split}
    \partial_g \overline{Q}_{\eex}(e_0) 
    &= \innprod{\chi_c}{\chi_g + \chi_{-g}} 
    - \innprod{\chi_{g_x}}{\chi_g + \chi_{-g}} - \innprod{\chi_{g_y}}{\chi_g + \chi_{-g}}\\
    &= \innprod{\chi_c - \rhostd}{\chi_g + \chi_{-g}}\\
  \end{split}
\end{equation*}
Finally recall that $\partial_g \overline{Q}_{\eex}(e_0) = -2b_{g,0}$ and 
$b_{g'+g,g'+h} = b_{g,h}$. After substituting this into the above expression 
and rearranging terms we have (nearly) proven:
\begin{theorem}\label{EulerFormCoeffs}
    Let $\eex = [\aff^2/G]$ and $\rho \in \Rep(G)$ satisfy
    $\eex^{[\rho]} \neq \emptyset$. Then for $a_g = \innprod{\rho}{\chi(g)}$
    we have
    \begin{equation*}
      Q_{\eex}(\rho) = 
        \sum_{(g,h) \in D} b_{g,h} (a_g - a_h)^2
    \end{equation*}
    where $b_{g,h} = \frac{1}{2}\innprod{\rhostd - \chi_c}{\chi_{g-h} + \chi_{h-g}}$.
    \begin{proof}
      The above calculations prove the claim when $\eex$ is rigid and 
      thus all that remains is to extend to the general 
      case. As such let $H$ denote the stabilizer group of the generic 
      point. Recall the quotient $G \to G/H$ induces an inclusion 
      of groups $\Irr(G/H) \subset \Irr(G)$ by dualizing, and thus 
      we will regard $\Irr(G/H)$ as a subgroup. First note 
      $g_x,g_y,c \in \Irr(G/H)$. Furthermore recall if 
      $\eex^{[\rho]} \neq \emptyset$ then necessarily $\rho \in \Rep(G/H)$. 
      In particular if $a_g \neq 0$ then $\chi(g) \in \Irr(G/H)$.
      Now assume $\chi(g),\chi(h) \notin \Irr(G/H)$. It follows $a_g, a_h = 0$
      since $\rho \in \Irr(G/H)$ and thus $b_{g,h}(a_g - a_h)^2 = 0$.
      On the other hand assume $\chi(g) \in \Irr(G/H)$ and 
      $\chi(h) \notin \Irr(G/H)$. In this case although it may 
      happen $a_g \neq 0$ observe that $\chi(g - h) \notin \Irr(G/H)$ 
      as $\Irr(G/H)$ is a subgroup. As such
      \begin{equation*}
        b_{g,h} = \frac{1}{2} \innprod{\rhostd - \chi_c}{\chi_{g - h} + \chi_{h - g}} = 0
      \end{equation*} 
      since both terms in the left component correspond to representations 
      contained in $\Irr(G/H)$. Thus $b_{g,h} (a_g - a_h)^2 = 0$ in 
      this case as well. In total the only terms which nontrivially 
      contribute to the sum are those when $\chi(g),\chi(h) \in \Irr(G/H)$.
      Finally the equality
      \begin{equation*}
      Q_{\eex}(\rho) = \sum_{(g,h) \in D} b_{g,h} (a_g - a_h)^2
      \end{equation*}
      may be established by passing to $\eexgr$.
    \end{proof}
\end{theorem}
Observe that every coefficient of $Q_{\eex}$ is calculated with the 
above procedure when $\eex$ is rigid. As such, in 
this case the Euler form may also be characterized as the unique quadratic 
form which satisfies $Q_{\eex}(\rho) = \chi(\Ocal_{\zz}, \Ocal_{\zz})$ 
whenever $\varrho(\zz) \simeq \rho$. However if $\eex$ is not rigid 
then the above equality is insufficient to uniquely specify $Q_{\eex}$. 
\begin{example}
  There is the equivalence of actions $(4;2,2) \simeq (2;1,1)$. Let 
  $\eex$ be the quotient stack of the first action and note
  $\eexgr$ corresponds to the second action. Upon setting 
  $Q_{\eex}^{\operatorname{rig}}$ to be the Euler form of $\eexgr$ it is easy to 
  calculate $Q_{\eex}^{\operatorname{rig}}(A_0,A_1) = 2(A_0 - A_1)^2$. 
  Upon taking a rigidification the induced 
  $\Irr(\mu_4/\mu_2) \subset \Irr(\mu_4)$ corresponds to the subgroup 
  $(2) \subset \mu_4$ so that $Q_{\eex} = 2(a_0 - a_2)^2$. Next set 
  $C(a_0,a_1,a_2,a_3) = 2(a_0 - a_2)^2 + 2(a_1 - a_3)^2$.
  Recall that for all $\rho \in \Rep(\mu_4/\mu_2) \subset \Rep(\mu_4)$ we 
  have $\rho \simeq a_0 \chi(0) + a_2 \chi(2)$. As such if $\eex^{[\rho]}$ 
  is nonempty then clearly $Q_{\eex}(a_0,a_1,a_2,a_3) = C(a_0,a_1,a_2,a_3)$.
\end{example}

\section{The Generalized McKay Graph of $\eex$} 

One additional construction which may be performed is the association of a 
graph to $\eex$. To be precise we will fix an abelian planar stack $\eex$ 
and let $Q_{\eex}$ denote its Euler form. Upon passing to a rigidification 
we will reduce to the case $\eex = \eexgr$ and define the associated graph 
for $\eex$ to be one and the same as the graph of $\eexgr$. 
We will now discuss the construction of the McKay graph of the subgroup 
$G \subset GL(2)$ corresponding to the action, which is denoted $\Gamma_G$. 
Recall,
\begin{definition}
    $\Gamma_G$ is the graph constructed as follows: 

    (i) $\operatorname{Nodes}(\Gamma_G) = \Irr(G)$
    
    (ii) Given two nodes $\chi(g), \chi(h)$ set 
    $n_{g,h} = \innprod{\chi_h \cdot \rhostd}{\chi_g}$.
    Then there is a unique directed edge $\chi(g) \rightarrow \chi(h)$ iff 
    $n_{g,h} > 0$, in which case the weight is given by $n_{g,h}$. Otherwise 
    no directed edges of $\Gamma_G$ start at $\chi(g)$ and terminate at 
    $\chi(h)$.
    
    (iii) If $n_{g,h} = n_{h,g}$ then both directed edges between 
    $\chi(g), \chi(h)$ are replaced by a single undirected edge of the same 
    weight.
\end{definition}
With this in mind we will continue with an analogous construction for $\eex$. 
By Theorem \ref{EulerFormCoeffs} the coefficients of $Q_{\eex}$ are given by 
\begin{equation*}
    b_{g,h} = \frac{1}{2}\innprod{\rhostd - \chi_c}{\chi_{g-h}+\chi_{h-g}}
\end{equation*}
Upon substituting this into the inner product and distributing across the 
sums we can see the terms of the sum give rise to independent contributions 
to the total coefficient. As such define 
\begin{equation*}
    B_{g,h} =\frac{1}{2} \innprod{\rhostd - \chi_c}{\chi_{g-h}}; ~ B_{h,g} = \frac{1}{2}\innprod{\rhostd - \chi_c}{\chi_{h - g}}
\end{equation*}
Clearly $B_{g,h}, B_{h,g} \in \frac{1}{2} \Z$ and 
$b_{g,h} = B_{g,h} + B_{h,g}$. In general 
$B_{g,h} \neq B_{h,g}$ and the most prominent exception to this rule is when 
$G \subset SL(2)$. As $\chi(g - h) \simeq \chi(g) \otimes \chi(h)^{\vee}$
and the left regular representation acts by orthogonal matrices we may show 
\begin{equation*}
    B_{g,h} = \frac{1}{2} \innprod{\chi_h \cdot (\rhostd - \chi_{c})}{\chi_{g}}
\end{equation*}
Define:
\begin{definition}
    The Generalized McKay graph of $\eex$, denoted $\Gamma_{\eex}$, is given 
    by the following data:

    (i) $\operatorname{Nodes}(\Gamma_{\eex}) = \Irr(G)$

    (ii) Given $\chi(g), \chi(h)$ such that $B_{g,h} \neq 0$ there is a 
    a directed edge $\chi(g) \rightarrow \chi(h)$ whose weight is given by 
    $2B_{g,h}$. Otherwise no directed edge of $\Gamma_{\eex}$ starting from 
    $\chi(g)$ and terminating at $\chi(h)$ exists. 
    
    (iii) If $B_{g,h} = B_{h,g}$ then both directed edges between 
    $\chi(g), \chi(h)$ are replaced by a single undirected edge of the same 
    weight.
\end{definition}
Next write
$$2B_{g,h} = \innprod{\chi_h \cdot \rhostd}{\chi_g} - \innprod{\chi_h \cdot \chi_c}{\chi_g}$$
If either $g$ or $h$ are nonzero then $g,h \neq c$. Thus as everything 
above is a character it is easy to see $2B_{g,h}$ is equal to exactly 
one of the terms on the right hand side, and thus we will partition 
$\Gamma_{\eex}$ into positive and negative parts.

Let $E_{g,h}$ denote the a directed edge starting at $\chi_g$ and ending at 
$\chi_h$ and assume $B_{g,h} \neq 0$. If $B_{g,h} > 0$ then $E_{g,h}$
is said to be a \emph{positive edge} and this will be denoted 
$E_{g,h} > 0$. If $B_{g,h} < 0$ then $E_{i,j}$ is said to be a 
\emph{negative edge} and this will be denoted $E_{g,h} < 0$.
Observe that if $B_{g,h} = B_{h,g}$ and $E_{g,h} > 0$ then necessarily 
$E_{h,g} > 0$. As such it is meaningful to declare an undirected edge to be 
positive or negative respectively.
\begin{definition}
    The positive part of $\Gamma_{\eex}$, denoted $\Gamma_{\eex}^+$ 
    is defined to be the spanning subgraph satisfying 
    $\operatorname{Edge}(\Gamma_{\eex}^{+}) = \{E_{g,h} > 0\}$.
    The negative part of $\Gamma_{\eex}$, denoted $\Gamma_{\eex}^-$ 
    is defined to be the spanning subgraph satisfying 
    $\operatorname{Edge}(\Gamma_{\eex}^{-}) = \{E_{g,h} < 0\}$.
\end{definition}
Upon comparing the construction of both graphs it is easy to prove:
\begin{prop}
    Let $\eex$ be a rigid abelian planar stack and $G \subset GL(2)$
    denote the subgroup associated to the action. Then there is a natural 
    isomorphism 
    $$\Gamma_{G} \simeq \Gamma_{\eex}^{+}$$
\end{prop}
As such the generalized McKay graph of $\eex$ may be used to recover the McKay 
graph of $G$.

\subsection{A Reinterpretation of the Dimension Formula}
We will conclude this section with one interpretation of the impact that 
positive and negative edges have on the geometry of $\eex$ and its Gilbert 
schemes when $\eex$ is rigid; To extend to arbitrary $\eex$ it is 
only necessary to recall that $\Gamma_{\eex}$ is defined by first passing to 
$\eexgr$ and there are natural isomorphisms 
$\eex^{[\rho]} \simeq \eexgr^{[\rho]}$.
Let $\rho \in \Rep(G)$ and define $a_g = \innprod{\rho}{\chi_g}$ so that 
$\rho \simeq \sum_{g \in g} a_g \chi(g)$. We have already proven 
\begin{equation*}
    \dim(\eex^{[\rho]}) = a_0 + a_{-c} - Q_{\eex}(\{a_g\})
\end{equation*}
and that $Q_{\eex}(\{a_g\}) = \sum_{(g,h) \in D} b_{g,h} (a_g - a_h)^2$.
By using that $b_{g,h} = B_{g,h} + B_{h,g}$ and each edge is uniquely 
positive or negative this last sum may be decomposed as
\begin{equation*}
    Q_{\eex}(\{a_g\}) = \sum_{E_{g,h} > 0} B_{g,h}(a_g - a_h)^2 - \sum_{E_{g',h'} < 0} \vert B_{g',h'} \vert (a_{g'} - a_{h'})^2
\end{equation*}
Substituting this into the dimension formula yields
\begin{equation}\label{DimPosNegContr}
    \dim(\eex^{[\rho]}) = a_0 + a_{-c} + \bigg[\sum_{E_{g',h'} < 0} \vert B_{g',h'} \vert (a_{g'} - a_{h'})^2\bigg] - \bigg[\sum_{E_{g,h} > 0} B_{g,h}(a_g - a_h)^2\bigg]
\end{equation}
To give the intended interpretation we will first apply this to a trivial 
case of Gilbert schemes. 

Let $X = \aff^2$ and $G = 0$. Note every representation of $G$
is of the form $m \RR$ and next fix a choice $\rho \simeq n \RR$. Then there 
is a natural isomorphism $[X/G]^{[n\RR]} \simeq X^{[n]}$ for all $n \in \N$. 
As this choice of stack corresponds to the $(0;0,0)$ action in turn $c = 0$ 
so that $a_0 = a_{-c}$. As $\rho \simeq n \RR$ it follows both are equal to 
$n$. Finally it is easy to verify in the case of Hilbert schemes that 
$Q_{\eex} = 0$ as $I \eex = \eex$ and the same reasoning above shows the 
integral evaluates to zero on this component. Thus from the dimension formula 
we obtain the expected dimension $\dim(X^{[n]}) = 2n$

Upon returning to more general Gilbert schemes we will instead view the 
terms of equation \ref{DimPosNegContr} as follows:
\begin{equation*}
\begin{split}
    \text{ExpectedDim}(\eex^{[\rho]}) &= a_0 + a_{-c}\\
    \text{InflationTerm}(\eex^{[\rho]}) &= \sum_{E_{g',h'} < 0} \vert B_{g',h'} \vert (a_{g'} - a_{h'})^2\\
    \text{DeflationTerm}(\eex^{[\rho]}) &= \sum_{E_{g,h} > 0} B_{g,h}(a_g - a_h)^2
\end{split}
\end{equation*}
As such negative edges will increase $\dim(\eex^{[\rho]})$ away from the 
expected dimension of $a_0 + a_{-c}$ and positive edges will act as a 
correction term to yield the correct quantity. Additionally the deflation 
term may be viewed as arising from representation theoretic influences as 
$\Gamma_{\eex}^{+} \simeq \Gamma_G$; It is a subtler 
question if there is a similar interpretation for the inflation term. 
Ideally there is a purely scheme theoretic interpretation 
for its existence, similar to how the McKay correspondence has a natural 
interpretation in terms of the geometry of blowups of certain quotient 
singularities.

\section{Explicit Values of Coefficients of Euler Forms}

This section will begin to give a classification schema of rigid 
abelian planar stacks by partitioning quotient stacks according 
to their Euler forms. This may be extended to arbitrary abelian planar stacks 
by defining the type of $\eex$ to be the type of $\eexgr$.
As such fix a rigid abelian planar stack $\eex = [\aff^2/G]$ and 
recall the elements $g_x, g_y \in G$ were defined by 
$\Span(x) \simeq \chi(g_x)$ and $\Span(y) \simeq \chi(g_y)$. In addition 
the canonical element is given by $c = g_x + g_y$.

\subsection{Trivial Type}
The first class which will be considered is when the action is 
sufficiently trivial. 
\begin{definition}
  $\eex$ is said to be of trivial type if $Q_{\eex} = 0$
\end{definition}
For more precise information on when this occurs we will prove:
\begin{lemma}
  $\eex$ is trivial type iff it satisfies one of the 
  three following conditions: (i) $g_x, g_y = 0$,
  (ii) $g_x = 0$ , or (iii) $g_y = 0$ .
  \begin{proof}
    Note that from direct calculation if any of the above conditions 
    hold then necessarily $Q_{\eex} = 0$ and so one direction is clear.
    For the other direction now assume $Q_{\eex} = 0$ and at least 
    one of $g_x,g_y$ is nonzero. Clearly $Q_{\eex} = 0$ iff all 
    coefficients are zero. As $\eex$ is rigid, with the natural 
    action of $G$ on $\Rep(G)$ there is the orthogonality relation 
    $Q_{\eex}(g \cdot \rho) = Q_{\eex}(\rho)$ for all $g \in G$.
    Thus if $b_{g,h}$ are the coefficients of $Q_{\eex}$ we have 
    $b_{g' + g, g' + h} = b_{g,h}$ for all $g',g,h \in G$ and thus 
    it suffices to consider only the coefficients $b_{g,0}$. Next by 
    Theorem \ref{EulerFormCoeffs} we have  
    $b_{g,0} = \frac{1}{2} \innprod{\chi_{g_x} + \chi_{g_y} - \chi_{c}}{\chi_{g} + \chi_{- g}} = 0$.
    As we have assumed not both of $g_x,g_y$ are zero this implies
    at least one of the characters in the left component 
    of the inner product is not equal to $\chi_{0}$. However if 
    both $\chi_{g_x}, \chi_{g_y}$ were nontrivial then we could choose 
    $g \in \{g_x, g_y\} \subset G$ to obtain a nonzero 
    coefficient. By relabeling if necessary we will assume 
    $\chi_{g_x} = \chi_{0}$. However as $\chi(g_x)$ 
    uniquely corresponds to $\chi_{g_x}$ this implies $g_x = 0$.
  \end{proof}
\end{lemma}

\subsection{Groupings}
As seen in the previous calculations the important factor for 
explicitly giving the coefficients of the Euler form is determining 
when elements of $\{\pm g_x, \pm g_y, \pm c\}$ coincide with each other.
To more simply state which classes coincide, we will use 
$(g_1,...,g_k)$ to indicate all $g_i \in \{g_x,...,-c\}$ are equal to each
other in $G$ and call it a \textbf{grouping}.

It is clear to each $\eex$ we may 
associate a grouping; For instance the grouping of trivial 
type stacks are $(g_x,-g_x,g_y,-g_y,c,-c)$. As this case 
has already been handled we will assume moving forward that 
$g_x, g_y \neq 0$. In particular this condition asserts we may disregard 
any grouping which contains both $g_x$ and $c$ in some cycle, and 
similarly for $(g_y,c)$,$(-g_x,-c)$,$(-g_y,-c)$. 
\begin{definition}
  A grouping is admissible if $g_x,g_y \neq c$ and it corresponds to a rigid stack.
\end{definition}
Excluding $g_x, g_y = c$ alone is quite restrictive among all possible 
groupings. For instance we may prove:
\begin{lemma}
  If a grouping is of the form $(*,*)(*,*,*,*)$ and admissible
  then necessarily it is $(c,-c)(g_x,-g_x,g_y,-g_y)$.
  \begin{proof}
    The entire grouping is determined by specifying the 2-cycle and 
    thus we will analyze cases accordingly. For 
    $(g_x,-g_x)(g_y,-g_y,c,-c)$ it follows $g_y = c$ and hence this grouping 
    is inadmissible. The remaining cases are handled identically.
  \end{proof}
\end{lemma}
In exactly the same fashion of analyzing each case it may be proven 
groupings whose associated partitions are 5+1, 4+2, 4+1+1, 3+2+1, or 
3+1+1+1 are always inadmissible. For example the grouping 
$(g_x)(c,-c)(-g_x,g_y,-g_y)$ is inadmissible since the 3-cycle implies 
$-g_y = g_y = -g_x$ and hence $g_x = -g_y = g_y = -g_x$, which is an 
impossibility for groupings arising from a stack. On the other hand we can 
completely characterize which groupings are admissible by analyzing the other 
partitions case by case:
\begin{prop}
  Up to a transposition of $g_x$ and $g_y$ the admissible groupings 
  are:
  \begin{itemize}
    \item Finite Classes: 
    \begin{itemize}
      \item (i) $(g_x,-g_x,g_y,-g_y)(c,-c)$
      \item (ii) $(g_x,-g_x)(g_y,-g_y)(c,-c)$
      \item (iii) $(g_x,g_y,-c)(-g_x,-g_y,c)$
      \item (iv) $(g_x,-g_x)(g_y,-c)(-g_y,c)$
      \item (v) $(g_x,g_y)(-g_x,-g_y)(c,-c)$
    \end{itemize}
    \item Skew Families: 
    \begin{itemize}
      \item $(i) (g_x,-g_y)(-g_x,g_y)(c,-c)$
      \item $(ii) (g_x,g_y)(-g_x,-g_y)(c)(-c)$
      \item $(iii) (g_x,-c)(-g_x,c)(g_y)(-g_y)$
    \end{itemize}
    \item Sporadic Classes: 
    \begin{itemize}
      \item (i) $(g_x,-g_x)(g_y)(-g_y)(c)(-c)$
      \item (ii)$(c,-c)(g_x)(-g_x)(g_y)(-g_y)$
    \end{itemize}
    \item General Type: $(g_x)(-g_x)(g_y)(-g_y)(c)(-c)$
  \end{itemize}
\end{prop}
As knowledge of which elements coincide completely determines the Euler 
form this is sufficient to be able to give numerical values for 
the coefficients of every Euler form of a rigid and abelian planar stack.
In each of the below calculations let 
$Q_{\eex} = \sum_{(g,h) \in D} b_{g,h} (a_g - a_h)^2$ so that by Theorem 
\ref{EulerFormCoeffs} we have 
\begin{equation*}
  b_{g,h} = \frac{1}{2} \innprod{\chi_{g_x} + \chi_{g_y} - \chi_{c}}{\chi_{g - h} + \chi_{h - g}}
\end{equation*}
We have already found the coefficients for the trivial type case. For the 
remaining groupings upon direct calculation with the above equality of 
elements we find the following explicit values of the coefficients;
In each case we will use $b_{g'+g,g'+h} = b_{g,h}$ and set $h = 0$
to simplify the expressions.

\underline{\textbf{Finite Classes}}:

\begin{equation*}
  (i): b_{g,0} = 
  \begin{cases}
    2,& g = \pm g_x\\
    0,& \text{else}\\
  \end{cases}
  ~(ii): b_{g,0} = 
  \begin{cases}
    ~1,& g = \pm g_x\\
    -1,& g= \pm c\\
    ~0,&\text{ else}\\
  \end{cases}
\end{equation*}

\begin{equation*}
  (iii):
  b_{g,0} = 
  \begin{cases}
    \frac{1}{2},& g = \pm g_x\\
    0,& \text{else}\\
  \end{cases}
  ~(iv):
  b_{g,0} = 
  \begin{cases}
    1,& g = \pm g_x\\
    0,& \text{else}\\
  \end{cases}
  ~(v):
  b_{g,0} = 
  \begin{cases}
    ~1,& g = \pm g_x\\
    -1,& g = \pm c\\
    ~0,& \text{else}\\
  \end{cases}
\end{equation*}

\underline{\textbf{Skew Families}}:

\begin{equation*}
  (i): b_{g,0} = 
  \begin{cases}
    1,& g = \pm g_x\\
    0,& \text{else}\\
  \end{cases}
  ~(ii): b_{g,0} = 
  \begin{cases}
    ~1,& g = \pm g_x\\
    -\frac{1}{2},& g = \pm c\\
    ~0,& \text{else}\\
  \end{cases}
  ~(iii): b_{g,0} = 
  \begin{cases}
    \frac{1}{2},& g = \pm g_y\\
    ~0,& \text{else}\\
  \end{cases}
\end{equation*}

\underline{\textbf{Sporadic Classes}}:

\begin{equation*}
  (i): b_{g,0} = 
  \begin{cases}
    ~1,&g = \pm g_x\\
    ~\frac{1}{2},&g = \pm g_y\\
    -\frac{1}{2},&g = \pm c\\
    ~0,&\text{else}\\
  \end{cases}
  ~(ii): b_{g,0} = 
  \begin{cases}
    ~\frac{1}{2},& g = \pm g_x, \pm g_y\\
    -1,& g = \pm c\\
    ~0,& \text{else}\\
  \end{cases}
\end{equation*}

\underline{\textbf{General Type}}:

\begin{equation*}
  b_{g,0} = 
  \begin{cases}
    ~\frac{1}{2},&g = \pm g_x, \pm g_y\\
    -\frac{1}{2},&g = \pm c\\
    ~0,&\text{else}
  \end{cases}
\end{equation*}

\section{The Classification}

This section will complete the proposed classification by analyzing the 
classes of stacks according to a given grouping. Some tools and 
observations which will prove to be useful are as follows:

In the doubly generated case we know that we may assume 
$G = \mu_m \oplus \mu_n$ and the action is of the form 
$(m; a,b) \oplus (n;c,d)$. Now let $e_1$ generate $\mu_m$ and 
$e_2$ generate $\mu_n$. By definition we have $e_2 = (\zeta^c x, \zeta^d y)$
for a fixed primitive $n$-th root of unity, $e_1 = (\omega^a x, \omega^b y)$
where $\omega$ is a primitive $m$-th root of unity, and $\zeta^D = \omega$ 
for $D = n/m$. As such for $(i,j) \in G$ we will calculate
\begin{equation*}
  (i e_1 + j e_2) \cdot (x,y) = (\zeta^{Dai + cj}x, \zeta^{Dbi + dj}y)
\end{equation*}
\begin{definition}
  The fundamental matrix of the $(m;a,b) \oplus (n;c,d)$ action is
  \begin{equation*}
  M_G =
  \begin{pmatrix}
    Da & c\\
    Db & d\\
  \end{pmatrix}
  \in M_{2,2}(\mu_n)
\end{equation*}
\end{definition}
It is easy to prove
\begin{prop}
  The $(m;a,b) \oplus (n;c,d)$ action corresponds to a rigid 
  stack iff $\ker(M_G) = 0$.
\end{prop}
As a consequence of this we have:
\begin{prop}\label{GroupConstraint}
  Suppose $G = \mu_n$ or $\mu_m \oplus \mu_n$ with $m \mid n$, and 
  that $\eex = [\aff^2/G]$ is a rigid stack. If 
  $k \cdot g_x, \ell \cdot g_y = 0$ and $L = \operatorname{lcm}(k,\ell)$ then 
  $n \mid L$.
  \begin{proof}
    In the cyclic case if the action is given by $(n;a,b)$ the hypothesis 
    amounts to requiring $ka, \ell b = 0 \modn$. On the other hand 
    we can calculate 
    $\zeta^{L} \cdot (x,y) = (\zeta^{aL} x, \zeta^{bL} y) = (x,y)$. As such 
    $\zeta^L$ stabilizes the generic point, but as $\eex$ is rigid 
    this implies $L = 0 \modn$.

    For the doubly generated case form the fundamental matrix.
    As $g_x = (a,c), g_y = (b,d)$ and $L \cdot g_x, L \cdot g_y = 0$
    it follows $(L,L)^T \in \ker(M_G)$. As the action was assumed to 
    be rigid we must have $(L,L)$ is the zero vector 
    and thus $L = 0 \modn$ as claimed.
  \end{proof}
\end{prop}

\subsection{Finite Classes}

We will now analyze groupings which are in the finite class.
To begin, we will assume $2g_x, 2g_y = 0$. This alone is enough to completely 
characterize the actions for which this holds. By Proposition 
\ref{GroupConstraint} and the fact 2 is prime we have 
$G = \mu_2, \mu_2^2$. With a case by case analysis of the possible actions 
it is easy to verify
\begin{prop}
  Suppose $\eex = [\aff^2/G]$ is rigid and $2g_x,2g_y = 0$. If 
  $\eex$ is not trivial type then the corresponding action up to 
  an automorphism of $\aff^2$ is 
  $(2;1,1)$, or $(2;1,0) \oplus (2;0,1)$, $(2;1,1) \oplus (2;0,1)$.
\end{prop}
It is clear the grouping associated to $(2;1,1)$ is $(g_x,-g_x,g_y,-g_y)(c,-c)$
and the grouping of the above doubly generated actions is instead 
$(g_x,-g_x)(g_y,-g_y)(c,-c)$ so in turn this characterizes the first two 
groupings in the finite class. Although we will not explicitly 
write the Euler forms for each action in the following subsections, in order 
to illustrate the procedure we will use the above systems of constraints for 
these cases. Clearly the Euler form of the $(2;1,1)$ action is 
$Q_{\eex}(a_0,a_1) = 2(a_0-a_1)^2$. On the other hand for the 
$(2;1,0) \oplus (2;0,1)$ action we have $g_x = (1,0), g_y = (0,1), c = (1,1)$ 
so the coefficients satisfy
\begin{equation*}
  b_{g,h} = 
  \begin{cases}
    ~1, g - h = (1,0), (0,1)\\
    -1, g - h = (1,1)\\
  \end{cases}
\end{equation*}
and hence
\begin{equation*}
  \begin{split}
    Q_{\eex}(a_{0,0}, a_{1,0}, a_{1,0}, a_{1,1}) 
    &= (a_{0,0} - a_{1,0})^2 + (a_{0,0} - a_{0,1})^2 - (a_{0,0} - a_{1,1})^2\\
    &~- (a_{1,0} - a_{0,1})^2 + (a_{1,0} - a_{1,1})^2\\
    &~+ (a_{0,1} - a_{1,1})^2\\ 
  \end{split}
\end{equation*}
Similarly the Euler form of the $(2;1,1) \oplus (2;0,1)$ action is
given by
\begin{equation*}
  \begin{split}
    Q_{\eex}(a_{0,0}, a_{1,0}, a_{1,0}, a_{1,1}) 
    &= -(a_{0,0} + a_{1,0})^2 + (a_{0,0} - a_{0,1})^2 - (a_{0,0} - a_{1,1})^2\\
    &~+ (a_{1,0} - a_{0,1})^2 - (a_{1,0} - a_{1,1})^2\\
    &~- (a_{0,1} - a_{1,1})^2\\ 
  \end{split}
\end{equation*}

\medskip
\underline{$(g_x,g_y,-c)(-g_x,-g_y,c)$}:
Immediately we may calculate $g_x = -c = -g_x - g_y = -2g_x$ and thus 
$3g_x, 3g_y = 0$. By Lemma \ref{GroupConstraint} and the fact that 3 is prime 
we can see $G = \mu_3, \mu_3^2$.
\begin{prop}
  The unique rigid stack whose grouping is $(g_x,g_y,-c)(-g_x,-g_y,c)$
  corresponds to the $(3;1,1)$ action.
  \begin{proof}
    Clearly the $(3;1,1)$ action gives rise to a rigid stack 
    and has the necessary grouping. As $g_x = g_y$ this is the only 
    possible cyclic action to consider. Now assume $\eex$ 
    corresponds to $(3;A,B) \oplus (3;C,D)$ and form the fundamental matrix.
    As $g_x = g_y$ it is given by 
    \begin{equation*}
      M_G = 
      \begin{pmatrix}
        A & C\\ A & C\\
      \end{pmatrix}
    \end{equation*}
    Then clearly $(-C,A)^T \in \ker(M_G)$ for any choice of $A,C$ 
    and thus no doubly generated with a grouping $(g_x,g_y,-c)(-g_x,-g_y,c)$
    can be rigid.
  \end{proof}
\end{prop}

\medskip
\underline{$(g_x,-g_x)(g_y,-c)(-g_y,c)$}:
It may be verified $2g_x = 0$ and moreso $g_y = -g_x - g_y$ implies 
$4g_y = 0$. By Lemma \ref{GroupConstraint} this implies 
$G = \mu_n, \mu_m \oplus \mu_n$ with $n \mid 4$. If $n = 2$ then 
$g_y = -c = c$ in which case the grouping is inadmissible so that
necessarily $n = 4$. 
In this cyclic case if the action is given by $(n;a,b)$ we have 
$g_x = a, g_y = b$ so the above equalities imply $a = 2$
and $4b = 0$ is a vacuous condition. By examination the actions
corresponding to a rigid stack are $(4;2,1), (4;2,3)$; As 
$2 = -2, 3 = -1$ (mod 4) by Lemma \ref{UnitRemoval} these are equivalent 
actions. 

For the doubly generated case suppose the action is of the form 
$(m; A,B) \oplus (n;C,D)$. As $n = 4$ and $m \mid n$ the possible 
choices of parameters are $m = 2, n = 4$ and $m = 4, n = 4$.
\begin{prop}
  If $G = \mu_m \oplus \mu_4$ with $m \mid 4$ then up to equivalence 
  $(2; 0,1) \oplus (4; 2,1)$ is the the unique action corresponding to a rigid stack whose grouping is $(g_x,-g_x)(g_y,-c)(-g_y,c)$.
  \begin{proof}
    In both cases clearly $g_x = (A,C), g_y = (B,D)$. The constraint 
    $2g_x = 0$ implies $C = 0, 2$, and it may be verified $g_y = -c$ implies
    $(2B,2D) = (-A,-C)$. 
    
    Now assume $m = 2$. By the fact that $2B = -A$ (mod 4) we can see $A = 0$ 
    (mod 2) so the fundamental matrix is given by
    \begin{equation*}
      M_G = 
      \begin{pmatrix}
        2A & C\\ 2B & D\\
      \end{pmatrix}
      =
      \begin{pmatrix}
        0 & C\\ 2B & D\\
      \end{pmatrix}
    \end{equation*}
    If $C = 0$ then clearly $M_G$ has a nontrivial kernel so assume $C = 2$.
    Upon doing so, in order for the matrix to have a trivial kernel it is 
    necessary and sufficient that $B,D$ be congruent to $1,3$ (mod 4) and 
    hence the action is of the form $(2; 0,1) \oplus (4; 2,D)$ for $D = 1,3$.
    As $(-2,-3) = (2,1) \in \mu_4^2$ by Lemma \ref{AutoRemoval} 
    we have $(2;0,1) \oplus (4;2,1) \simeq (2;0,1) \oplus (4;2,3)$.

    Now assume $m = 4$. By the equality $(2B,2D) = (-A,-C)$ the 
    fundamental matrix is given by
    \begin{equation*}
      M_G = 
      \begin{pmatrix}
        A & C\\ B & D\\
      \end{pmatrix}
      =
      \begin{pmatrix}
        2B & 2D\\ B & D\\
      \end{pmatrix}
    \end{equation*}
    from which clearly $(1,-2)^T \in \ker(M_G)$ and hence the corresponding 
    cannot be rigid for any choice of $A,B,C,D$.
  \end{proof}
\end{prop}

\medskip
\underline{$(g_x,g_y)(-g_x,-g_y)(c,-c)$}: 
It is easy to see $g_x = g_y$ and $2c = 0$
imply $4g_x, 4g_y = 0$. Thus by Lemma \ref{GroupConstraint} we can see
$G = \mu_4, \mu_2 \oplus \mu_4, \mu_4^2$. It is easy to verify when 
$G = \mu_4$ the only action up to equivalence is $(4;1,1)$. As such 
assume $G = \mu_m \oplus \mu_4$ where $m \mid 4$ and the action is 
$(m;A,B) \oplus (n; C,D)$. Clearly $g_x = g_y$ so that $A = B, C = D$
and hence if $d = n/m$ then the fundamental matrix is
\begin{equation*}
  M_G = 
  \begin{pmatrix}
    dA & C\\ dB & D\\
  \end{pmatrix}
  =
  \begin{pmatrix}
    dA & C\\ dA & C\\
  \end{pmatrix}
\end{equation*}
and independent of the parameters this always has a nontrivial kernel. Thus 
we have proven:
\begin{prop}
  Up to equivalence, the only action whose corresponding stack is 
  rigid and has a grouping $(g_x,g_y)(-g_x,-g_y)(c,-c)$ is 
  the $(4;1,1)$ action.
\end{prop}

\subsection{Skew Families}

Compared to the previous cases the treatment for the remaining groupings will 
be comparitively brief. This is largely due to the fact that for each of them 
there are infinitely many rigid stacks whose grouping is given by 
the one under consideration for both the cyclic and doubly generated case. 
As a consequence we will instead give descriptions of the stacks only in the 
cyclic case and return to the doubly generated case in the sequel. In line 
with this we will now fix a choice of action $(n;a,b)$ so that 
$g_x = a, g_y = b, c = a + b$, and relabel the groupings accordingly.

\medskip
Define the $n$-th \textbf{diagonal action}, or equivalently the $A_{n - 1}$, 
action as the $(n;1,-1)$ action and denote its corresponding quotient stack by 
$\mathcal{D}_n$. Similarly define the $n$-th \textbf{antidiagonal} action to 
be the $(n;1,1)$ action and denote the quotient stack by $\Acal_n$. 
We will say $\eex = [\aff^2/G]$ is diagonal type if there 
exists $n \in \N$ such that $\eexgr \simeq \mathcal{D}_n$ and 
antidiagonal type if $\eexgr \simeq \mathcal{A}_n$. Already in the 
literature actions of this form have been studied; Diagonal actions 
are precisely those cyclic actions which factor through $SL(2)$, and 
antidiagonal actions have been treated in works such as 
\cite{CastroRoss2019}.
Next for simplicity fix $\eex = [\aff^2/\mu_n]$. 

Assume $\eex$ is rigid and its grouping is $(a,-b)(-a,b)(c,-c)$. 
Clearly $c = 0$ and 
the corresponding action is $(n;a,-a)$. As $\eex$ is rigid necessarily
$a$ is a unit modulo $n$ so that by an application of Lemma \ref{UnitRemoval} 
it is clear $\eex \simeq \mathcal{D}_n$. Conversely given the $(n;1,-1)$ 
action in order for it to have a grouping equal to $(a,-b)(-a,b)(c,-c)$ it is 
necessary that $m > 2$ and this proves:
\begin{prop}
  Suppose $\eex = [\aff^2/\mu_n]$ is rigid and $n > 2$. Then its 
  grouping is \newline 
  $(a,-b)(-a,b)(c,-c)$ iff $\eex \simeq \mathcal{D}_n$.
\end{prop}
By the exact same argument we may also prove
\begin{prop}
  Suppose $\eex = [\aff^2/\mu_n]$ is rigid and $n > 4$. Then its 
  grouping is \newline 
  $(a,b)(-a,-b)(c)(-c)$ iff $\eex \simeq \mathcal{A}_n$.
\end{prop}
The final case of interest is the grouping $(a,-c)(-a,c)(b)(-b)$.
In this case the condition $a = -c$ implies $-2a = b$. Hence if 
$\eex$ is rigid and its grouping is given as above then in turn the 
corresponding action is equivalent to $(n;a,-2a)$. As we have assumed 
$\eex$ is rigid in turn $a$ is a unit so by Lemma \ref{UnitRemoval}
the action is equivalent to $(n;1,-2)$. Let $\BB_n$ be the stack 
corresponding to correspond to the $(n;1,-2)$ action, so that if we say
$\eex$ is \textbf{(1,-2)-type} if $\eex \simeq \BB_n$ for some $n$ we have 
proven:
\begin{prop}
  Suppose $\eex = [\aff^2/\mu_n]$ is rigid and $n > 5$. Then its 
  grouping is 
  \newline $(a,-c)(-a,c)(b)(-b)$ iff $\eex \simeq \BB_n$.
\end{prop}
Of particular note is that the actions $(n;1,-1), (n;1,1), (n;1,-2)$ 
give canonical representatives for stacks of a given type, and in 
such a way that depends only on the single parameter $n$. Hence each of 
these three groupings gives rise to a linear family of stacks. The remaining
classes will not have this feature.

\subsection{Sporadic Classes and General Type}

The treatment of these groupings will be perhaps the most brief. In part 
this is due to the fact that unlike the previous three where there was a 
single parameter which parameterized all infinitely many cyclic stacks,
now there is no such tool for these infinite classes. Instead they 
``sporadically" appear which is the motivation of the names of these classes.

To be more precise, in the cyclic case the grouping $(a,-a)(b)(b)(c)(-c)$
is given by any action $(n;a,b)$ where $2a = 0 \modn$ and $\gcd(a,b) = 1$;
Such a stack will be called \textbf{type $2A$}. Similarly 
the grouping $(a)(-a)(b)(-b)(c,-c)$ will correspond to actions of the form 
$(n;a,b)$ where $\gcd(a,b) = 1$ and $2c = 0 \modn$; Stacks arising in this 
fashion will be referred to as \textbf{type $2C$}. We may define type 2B stacks
but this will give rise to a class which is equivalent to type 2A by a 
transposition of $x$ and $y$ in $\C[x,y]$.

Finally $\eex$ will be said to be \textbf{general type} if its grouping is 
$(g_x)(-g_x)(g_y)(-g_y)(c)(-c)$; Given a "generic" action it is typical 
the corresponding elements are all distinct to each other, and thus the name 
represents how commonly stacks of this form arise. As will be discussed in the 
next section a fruitful way to view these and the other families is 
by considering all actions simultaneously.

\subsection{Asymptotic Behavior}

Following the above classification, one final remark in this direction is on 
the nature of $(n;a,b)$ actions when considered as families by fixing a choice 
of $a,b$ and letting $n$ vary. To simplify matters we will only consider 
actions where $a = 1$ and define $\eex_b(n)$ to be the stack corresponding
to the $(n;1,b)$ action. As $a = 1$ it is the case that $\eex_b(n)$ is 
rigid for each pair $b,n$, and if $(b',n') \neq (b,n)$ then 
$\eex_b(n) \not\simeq \eex_{b'}(n')$. As such there is a bijection of 
isomorphism classes $\{\eex_b(n)\} \leftrightarrow \N^2$ given by 
$\eex_b(n) \leftrightarrow (b,n)$. By plotting these points and 
marking which class the corresponding stack lies in we obtain the picture:

\begin{center}
\begin{tikzpicture}[scale=0.4]
 
  \draw[step=1cm, gray!60, very thin] (0,0) grid (10,9);
  \draw[->] (-0.5,0) -- (10.5,0) node[right] {$n$};
  \draw[->] (0,-0.5) -- (0,9.5) node[above] {$b$};
 
 
  \foreach \n in {2,...,10} {
    \fill[black] (\n,0) circle (3pt);
  }
 
  \foreach \n in {3,...,10} {
    \fill[green!60!black] (\n,\n-1) circle (3pt);
  }
 
  \foreach \n in {5,...,10} {
    \fill[orange] (\n,1) circle (3pt);
  }
 
  \foreach \n in {6,...,10} {
    \fill[red] (\n,\n-2) circle (3pt);
  }
 
  \foreach \p in {(6,3),(8,4),(10,5)} {
    \fill[blue] \p circle (3pt);
  }
 
  \foreach \p in {(6,2),(8,3),(10,4)} {
    \fill[purple] \p circle (3pt);
  }
 
  \foreach \p in {(2,1),(3,1),(4,1),(4,2)} {
    \fill[gray] \p circle (3pt);
  }
 
\end{tikzpicture}
\end{center}

Here black points are trivial type, grey are finite type, blue points are 
type $2A$, purple points are type $2C$, green are diagonal type, orange are 
antidiagonal type, and red are (1,-2) type. All other points which correspond 
to a nontrivial stack are general type. In this way each of the diagonal,
antidiagonal, and (1,-2) type families correspond to special lines in the 
plane. The above picture on the other hand is slightly misleading 
in terms of the nature of types $2A$ and $2C$. 

To justify the names of the latter classes of stacks, now we will fix a choice 
of $a_0,b_0 \in \Z$ such that $\gcd(a_0,b_0) = 1$. Then for each $n \in \N$ 
the stacks $\eey_n$ which correspond to the $(n;a_0,b_0)$ action are always 
rigid. However it is clear that since $a_0,b_0$ are fixed, only 
finitely many $\eey_n$ can be of type $2A$ and $2C$. In fact for $n\gg 0$ it is 
easy to see $\eey_n$ is of general type. As such types $2A, 2C$ lie 
``orthogonally" to the previous classes, whereas general type stacks 
constitute a dense open subset. 

\section{Excisable Subsets of Colored Diagrams}

We will now commence with the final examination of the paper. The following 
sections will culminate in an application of the results of the 
previous sections towards the combinatorics of Young diagrams. As such 
we will begin by recalling their definitions. In what follows 
the $A_{n-1}$ action will stand in for the $(n;1,-1)$ action due 
to the connection to the McKay correspondence, and a \textbf{monomial subscheme} 
is a subscheme $\xi \subset \aff^2$ such that $I(\xi)$ is a finite colength 
monomial ideal; Note that in isotypic coordinates monomial subschemes are always 
$G$-equivariant. 

As a final remark before continuing on, in 
contrast to existing literature where $(n;1,-1)$ actions are referred to as
antidiagonal actions, it is our contention it is more fitting to call them
diagonal actions. This is due to the fact that many of the important 
deductions on the combinatorics of diagrams colored with respect to 
$A_{n-1}$ actions follow from the fact that colors are preserved along the 
diagonals of the diagram.

\subsection{Colored Young Diagrams}

Let $x,y$ be an isotypic basis of $\C[x,y]$. Let $M$ denote the set of 
monomials in $x,y$ and endow it with the deglex ordering. Form the bijection 
$f: M \rightarrow \N^2$ given by taking a monomial to the 
multi-index of its exponents. Next, for a monomial ideal $I$ define the 
subset $Y(I) \subset \N^2$ which consists of those monomials not contained in 
$I$. Under the above bijection $Y(I)$ will be regarded as a subset of $\N^2$ 
from which $Y$ is easily seen to be a Young diagram. As we will only consider 
diagrams of this form the qualifier of Young will be omitted.

Primarily diagrams will be used to keep track of the group action by means 
of colorings. 
\begin{definition}\label{def:CanonicalColoring}
  The canonical coloring of the action of $G$ is the function 
  \begin{equation*}
    \mathcal{C}_G: \N^n \rightarrow G
  \end{equation*}
  given by taking $x^I$ to the unique $g \in G$ such that 
  $$ \Span(x^I) \simeq \chi(g) \in \Irr(G)$$ 
\end{definition}
\begin{remark}
Giving the coloring in terms of the elements of $G$ is 
easier to state, but often it will be given in terms of actual colors if 
diagrams are used in a calculation. 
\end{remark}
For a given diagram the inclusion $Y \subset \N^2$ gives rise to a coloring 
compatible with the canonical coloring of the action. The pair 
$(Y, \mathcal{C}_G)$ will be called a colored diagram, 
and its color type is the sequence $(a_g)_{g \in G}$ where 
$a_g = \#(\CC_G^{-1}(g) \cap Y)$. The connection between colored 
diagrams and $G$-subschemes is straightforward to prove:
\begin{lemma}
  Let $\rho \in \Rep(G)$ and assume its isotypic decomposition is
  $\rho \simeq \sum_{g \in G} a_g \chi(g)$. Let $\xi \in \eex^{[\rho]}$
  be a monomial subscheme and $Y = Y(I(\xi))$ be colored via $\CC_G$. 
  Then the color type of $Y$ coincides with the coefficients of the isotypic 
  decomposition above. Moreso this defines a bijection between colored 
  diagrams and monomial subschemes.
\end{lemma}
In terms of the bijection the diagram associated to $\xi$ will be denoted 
$Y(\xi)$ with no reference to the ideal. Next, we will define a distinguished 
class of subsets of Young diagrams called nets.
\begin{definition}\label{def:Nets}
    Let $Y \subset \N^n$ be a diagram and $e_1 = (1,0), e_2 = (0,1)$. For 
    $p \in Y$ the subset
    \begin{equation*}
      H(p) = \{p + n e_i \in Y \mid n \in \N\}
    \end{equation*}
    will be called the hook of $p$, and the subset
    \begin{equation*}
        N(p) = \{p + v \in Y \mid v \in \N^2\}
    \end{equation*}
    will be called its net. $p$ will be referred to as the eye of either its 
    hook or its net.
\end{definition}
Hooks of a point in a diagram are already familiar tools to 
combinatorialists. The purpose of extending the definition to also include 
everything ``up and to the right" of $p$ is due to the fact that nets 
correspond to monomial subideals of a given monomial ideal. In addition, by 
translating the eye of a net to $(0,0) \in \N^2$ we obtain a diagram.

\subsection{Border Strips}

Let $Y \subset \N^2$ be a diagram. 
\begin{definition}\label{def:Border}
  The border of $Y$ is the subset 
  \begin{equation*}
    B(Y) = \{p \in Y \mid p + (1,1) \notin Y\}
\end{equation*}
$d(Y) := \# B$ will be called the diameter of Y and the length of the border.
\end{definition} 
A connected subset of $B(Y)$ will be called a \textbf{border strip}. A binary 
sequence can be associated to a border strip in order to more precisely refer to 
$B$: 

Enumerate $B \subset \N^2$ as 
$\{(a_i, b_i)\}_{i=1}^{d(Y)}$ such that either 
$a_i > a_{i + 1}, b_i = b_{i + 1}$ or $a_i = a_{i + 1}, b_i > b_{i + 1}$. 
Such an enumeration is possible as every $p \in B$ satisfies 
$p + (1,1) \notin Y$. Set $S_0 = \emptyset$. Assuming 
$S_j = \{c_i\}_{i=1}^j$ has already been obtained, $S_{j+1}$ is then 
defined as $S_j \cup \{c_{j+1}\}$ where $c_j = 0$ if $a_j = a_{j+1}$, and 
$c_j = 1$ if $b_{j} = b_{j + 1}$. The \textbf{staircase sequence} of Y is
$S_{d(Y) - 1}$, and may be denoted as either $S(Y)$ or $S(B)$ depending on
which is being emphasized.

Conversely for each staircase sequence there is a canonical diagram which 
may be associated to it. To see this, enumerate $S_j = (c_i)_{i = 0}^{m - 1}$ 
and let $j = \sum_{i = 0}^{m - 1} c_i 2^i$.
\begin{lemma}
  For each $0 \leq j < 2^{n - 1}$ there exists a unique 
  Young diagram $Y_j$ such that $S(Y_j) = S_j$. 
  \begin{proof}
    Let $e = \#\{c_i = 0\}, f = \# \{c_i = 1\}$. Inductively define 
    $B_j \subset \N^2$ by $B_j = \{b_i\}$ with $b_0 = (0, f)$ and 
    $b_{i + 1} = b_i + (1,0)$ if $c_i = 0$ and $b_{i + 1} = b_i - (0,1)$ if 
    $c_i = 1$. $Y_j$ then consists of all points $p \in \N^2$ with 
    $p \leq b_i$ for some $b_i \in B_j$. It is readily shown $Y_j$ is a 
    diagram, $B(Y_j) = B_j$, and $S(Y_j) = S_j$.
  \end{proof}
\end{lemma}

\subsection{Excisable Subsets}

In this section $\eex = [\aff^2/G]$ is a fixed abelian planar stack
and $\CC$ is the canonical coloring of the action of $G$. Let 
$\xi \in \eex^{[\rho]}$ be a monomial subscheme, $Y = Y(\xi)$, and $B$ be the 
border of $Y$. To simplify the language later we will define the following. 
\begin{definition}
  $\GG_{\rho} \subset \eex^{[\rho]}$ is the subscheme whose closed 
  points are in a bijective correspondence with $\xi \in \eex^{[\rho]}$
  such that $\operatorname{Supp}(\xi) = 0$.
\end{definition}
Of course $\GG_{\rho}$ is the analogue of punctual Hilbert schemes;
The fact that it is a subscheme follows from expressing it as a 
closed subscheme of some Grassmannian, and hence it will be 
referred to as the $G$-equivariant Grassmannian of $\rho$. Next define
\begin{definition}\label{def:Disentanglement}
  Let $\eta \in \eex^{[\rho]}$ be supported at the origin, and 
  $\Xi \subset \aff^2 \times T$ be a $G$-flat family over some connected 
  and reduced scheme $T$. $\Xi$ is said to be a disentanglement if there 
  exists a nonempty open subset $U \subset T$ such that for all $t \in U$:
  (i) $\Xi_t = f_t \cup \zeta_t$ with $f_t \in \Delta^{[m\RR]}$ for some 
  $m > 0$ independent of $t$, and (ii) $\zeta_t \in \GG_{\rho - m \RR}$.
\end{definition}
Any qualifiers on $\Xi$ will be included in the declaration of the 
disentanglement. A monomial subscheme of a disentanglement will be 
called a \emph{realization} of $\Xi$. The integer $m$ will be called the 
\emph{gauge} of the disentanglement. The most frequently used class of 
disentanglements are those where $\zeta_t$ is constant as these have simple 
interpretations in terms of diagrams. 
\begin{definition}\label{def:SID}
  A disentanglement $\Xi$ will be called stable if there exists $\zeta$ such 
  that $\Xi_t = f_t \cup \zeta$ for all $t \in U$. $\zeta$ will be referred 
  to as the base of the disentanglement. 
\end{definition}

Let $\Xi \subset \aff^2 \times T$ be a disentanglement realized by $\xi$. 
The most important class of border strips for the following analysis are 
those which may be algebraically removed from $Y$:
\begin{definition}\label{def:ExcBorderStrip}
    $B$ is an excisable subset of $Y$ if there exists 
    a disentanglement $\Xi \subset \aff^2 \times T$ which is realized by 
    $\xi$, and $t \in T$ such that $\Xi_t = \zeta \cup f$ where 
    $f \subset \Delta$ and $\zeta$ is a monomial subscheme such that 
    $Y - B = Y(\zeta)$.
\end{definition}
If $B$ is a border strip it will simply be referred to as an excisable border 
strip. For example consider the $A_2$ action and the families of ideals 
  $$(x^3 - t,y), ~(x,y^3 - t), ~(x^2 - ty, y^2 - tx, xy - t^2), ~(x,y^3-t)$$ 
where $t \in \aff^1$. For each of them by taking the zero locus, the 
resulting family of subschemes is an SID of gauge 1. Pictorially they 
excise the following border strips
\begin{center}
  \begin{ytableau}
    *(green) \circ\\
    *(red) \circ\\
    *(white)\circ\\
  \end{ytableau}
  ~
  \begin{ytableau}
    \none\\
    *(red)\circ & *(white)\circ\\
    *(white) & *(green)\circ\\
  \end{ytableau}
  ~
  \begin{ytableau}
    \none\\
    *(red)\circ\\
    *(white)\circ & *(green)\circ\\
  \end{ytableau}
  ~
  \begin{ytableau}
    \none\\
    \none\\
    *(green)\circ & *(red)\circ & *(white)\circ\\
  \end{ytableau}
\end{center}
One clear obstruction to being able to deform away a border strip with a 
disentanglement is the failure to have the correct color type. As 
this is an important obstruction we will say $B$ is \emph{regularly colored}
if its color type is $(n,...,n)$ for some $n$. In fact the following 
argument will also prove that $n$ is equal to the gauge of the corresponding 
disentanglement.
\begin{lemma}
  If $B$ is an excisable border strip then $B$ is regularly colored.
  \begin{proof}
    Let $\Xi$ be the disentanglement as above such that 
    $\Xi_t = f \sqcup \zeta$ with $f \subset \Delta$. By Proposition 
    \ref{DeltaRepCase} it follows $\varrho(f) \in \N \cdot \RR$. As such 
    $\varrho(\Xi_t) \simeq \varrho(\zeta) + m\RR$ for some $m$ so that  
    $\varrho(\xi) \simeq \varrho(\zeta) + m\RR$ by
    Proposition \ref{GFlatLimit}. Thus $\varrho(\zeta) = \varrho(\xi) - m\RR$. 
    Finally as $Y(\zeta) = Y - B$, the claim that $B$ is regularly colored 
    follows by translating the above calculation into color types of the 
    corresponding diagrams. 
  \end{proof}
\end{lemma}
As such being regularly colored is a necessary condition, and in fact this is 
also sufficient. 
\begin{prop}\label{ExcisableSubsets}
  Let $Y$ be a diagram and $B \subset B(Y)$ be a regularly colored subset 
  such that $Y - B$ is a diagram. Then $B$ is excisable.
  \begin{proof}
    For simplicity we will assume the color type of $B$ is $(1,...,1)$.
    By assumption $Y - B$ is a diagram. Let $\xi$ be the 
    subscheme corresponding to $Y - B$ and $\eta$ the subscheme to $Y$. Let 
    $f = \Orb_G((1,1))$ and set $\Xi_1'(t) = \xi \cup (t \cdot f)$ for 
    $t \in \C^*$ using the standard torus action. It can 
    be verified $\Xi_1'$ is a $G$-flat family. Moreso as $t \rightarrow 0$ 
    the orbit limits to an attracting point, and hence a $G$-flat 
    limit exists. As such the family may be extended to 
    $\Xi_1 \subset \aff^2 \times \aff^1$ such that $\Xi_1(0) \in \GG_{\rho}$.
    
    As $G$-flat limits preserve representation type, necessarily $\Xi_1$ has 
    constant representation type which is isomorphic
    to $\rho = \varrho(\eta)$ so $\Xi'(0) \in \GG_{\rho}$. 
    By Propositions 3.10 and 3.12 of \cite{Maclagan_Smith_2010} we know
    $\GG_{\rho}$ is rationally connected by using the induced families from 
    positive and negative significant arrows. As such there exists a 
    family $\Xi_2 \subset \GG_{\rho} \times C$ for some connected curve 
    curve $C$ which deforms $\Xi_1(0)$ to some $\eta = \Xi_2(t_0)$ so the 
    concatenation $\Xi = \Xi_1 \cup \Xi_2$ is a disentanglement. As 
    $\Xi(t_0) = \eta$ and $\Xi(1) = \xi \cup t \cdot f$ the 
    condition for $B$ to be excisable has been met. 
  \end{proof}
\end{prop}

\subsection{Border Strips of $A_n$ Actions}

Unfortunately Proposition \ref{ExcisableSubsets} does not give any 
explicit information on which subsets are excisable for a given coloring; In 
general it is difficult to give a satsifactory answer to this. For example, 
consider the $(5; 1, 3)$ action. By examination, of the 16 staircase 
sequences of length 4 the only corresponding border strips which are regularly 
colored are
\begin{center}
  \begin{ytableau}
    *(red)\\
    *(Fuchsia)\\
    *(green)\\
    *(yellow)\\
    *(white)\\
  \end{ytableau}
  \begin{ytableau}
    \none\\
    \none\\
    *(green)\\
    *(yellow) & *(Fuchsia) & *(white)\\
    \none & \none & *(red)\\
  \end{ytableau}
  \begin{ytableau}
    \none\\
    \none\\
    \none\\
    \none\\
    *(white) & *(green) & *(red) & *(yellow) & *(Fuchsia)\\
  \end{ytableau}
\end{center} 
Due to these general complications we will restrict our attention to 
excisable border strips of $A_n$ actions. In what follows $\CC_n$ denotes 
the canonical coloring of the $A_{n - 1}$ action, and define an operation on 
binary sequences $S = (c_i)_{i = 1}^m$ called \textbf{flips}.
\begin{definition}
  For each $j$ and $S$ define the sequence $(c_i^*)$ by $c_i^* = c_i$ if 
  $i \neq j$ and $c_i^* = 1$ if $i = j$. Then $\ff_j^+(S)$ is given by $S$ 
  if $c_j = 1$ and $\ff_j^{+}(S) = (c_i^*)$ otherwise. 
\end{definition}
For a binary sequence $S$ let $\ff_j^{+}(Y)$ be the diagram associated to 
$\ff_j^{+}(S)$ and similarly define $\ff_j^{+}(B_i)$. 
\begin{prop}\label{RegColoring}
    Fix $n \in \N$, a binary sequence $S_i$, let 
    $Y_i$ be the minimal diagram of $S_i$, and set $B_i = B(Y_i)$. Then 
    $B_i$ is regularly colored with respect to $\CC_n$.
    \begin{proof}
        The argument will proceed on induction with respect to the 
        number of ones in $S_i$. If $S_i = (0,...,0)$ then 
        $Y_i = \{(i,0) \mid 0 \leq i \leq n - 1\}$ from which the claim 
        is easily verified. Next choose $i \neq 2^{n - 1} - 1$, 
        form $Y_i$ and choose $j$ such that $c_j$ is the last entry of 
        $S_i$ which is zero. Consider the sequence $\ff_j^{+}(S_i)$. 
        By giving $B_i = \{p_k\}$ a linear order as above we can define 
        $B_i^* = \{p_k^*\} \subset \N^2$ by setting $p_k^* = p_k$ if 
        $k < j$, and $p_k^* = p_k - (1,1)$ if $k \geq j$. 
        After changing $0$ to $1$ in the sequence it may be the case
        $p_i^*$ has negative coordinates. 
        
        To circumvent this recall that the 
        coloring is a semigroup homomorphism $\CC_n: \N^2 \rightarrow \mu_n$
        and thus may be extended to a homomorphism 
        $\CC_n: \Z^2 \rightarrow \mu_n$. Note the extension also 
        satisfies $\CC_n(i,j) = \CC_n(i + N, j + N)$ for all $N \in \Z$. As 
        such $\CC_n(p_k^*) = \CC_n(p_k)$ for all $k$ so $B_i$ and $B_i^*$ 
        have the same color type. On the other hand by forming 
        $\ff_j^{+}(B_i)$ from $\ff_j^{+}(S_i)$ it is not hard to see 
        $\ff_j^{+}(B_i) = B_i^* + v$ for some $v \in \Z^2$. As the effect of 
        translations on the color type is to permute the coordinates 
        necessarily $B_{i}^*$ is regularly colored iff $\ff_j^{+}(B_i)$ is 
        regularly colored. Finally since $\ff_j^{+}(S_i)$ 
        has less zeroes than $S_i$ by our inductive hypothesis 
        $\ff_j^{+}(B_i)$ is regularly colored, whence $B_i$ is as well.
    \end{proof}
\end{prop}

\section{Cores of a Group Action}

From experience, the notion of core is crucial to the 
geometry of Gilbert schemes and to the combinatorics of colored diagrams. In 
general core refers to a diagram which is minimal with respect to some 
property. Let $\eex = [\aff^2/G]$ be an abelian planar stack with 
canonical coloring $\CC$. Let $Y$ be a diagram colored with respect to $\CC$. 
We will now recall some definitions and results in order to precisely
specify the properties under consideration. Let $M$ be a monomial ideal 
with minimal generators $\{x^{p_i} y^{q_i}\}_{i=0}^{n}$ written such that 
$p_0,q_n = 0$, $p_i > p_{i + 1}$, and $q_i < q_{i+1}$. Let $Y = \N^2 - M$ be 
its associated diagram.

For the convenience of the reader we will include the following definitions,
the first of which may be found in section 3 of \cite{Maclagan_Smith_2010}. 

$\bullet$
An arrow of $M$ is a 
pair $(i;u,v)$ such that $\CC(p_i,q_i) = \CC(u,v)$ and $(u,v) \in Y$. 
$(i;u,v)$ is positive if $u > p_i$, nonnegative if $u \geq p_i$, and 
non-positive if $v \geq q_i$; Moving forward $(i;u,v) > 0$ will indicate an 
arrow is positive and we will similarly use inequalities for the other two 
classes. If $(i;u,v) \geq 0$ and $i > 0$ then it is significant if 
$(u + p_{i - 1} - p_i, v) \in M$. If $(i;u,v) \leq 0$ and $i < n$ then 
it is significant if $(u,v + q_{i + 1} - q_i) \in M$. Finally 
$T(M)$ denotes the collection of significant arrows of $M$.

$\bullet$
Let $\xi = \mathbb{V}(M) \subset \aff^2$, set $\rho = \varrho(\xi)$
and form the isotypic decomposition $\rho = \sum_{g \in G} a_g \chi(g)$.
Recall $(a_g)$ is the color type of $Y$ so it is meaningful to define the
function on diagrams $t(Y) = \dim(\eex^{[\rho]})$. It was already shown the 
right hand side depends only on $\rho$ and hence the left hand side depends 
only on the color type of $Y$.

$\bullet$
Recall the hook length of $p \in Y$ is $h(p) = \#H(p)$.
In addition for $t \in \N$, we will say $Y$ is a $t$-core if $t \nmid h(p)$ 
for all $p \in Y$. 

With these definitions in mind, the 4 classes of cores are:
\begin{definition}\label{def:CoreTypes}
  Y is...
  \begin{itemize}
    \item an $n$-core if the associated partition is an $n$-core partition
    \item a B-core if Y contains no excisable border strips 
    \item a G-core if $t(Y) = 0$
    \item an MS-core if $T(M) = \emptyset$
  \end{itemize}
\end{definition}
The different classes of cores have the following relationships:
\begin{prop}
  Let $n = \# G$. There is a diagram of implications:
\[\begin{tikzcd}
	n & B & G & MS
	\arrow[from=1-1, to=1-2]
	\arrow[from=1-3, to=1-2]
	\arrow[tail reversed, from=1-3, to=1-4]
\end{tikzcd}\]
\begin{proof}
  \underline{G $\rightarrow$ B}: Assume $Y$ is not a $B$-core so there exists a 
  disentanglement realized by $\xi$. Let $\zeta \sqcup f$ be a fiber 
  of the disentanglement such that $\zeta$ is a monomial subscheme and 
  $f \subset \Delta$. By constancy of dimension it follows 
  $t_{\eex}(\xi) = t_{\eex}(\zeta \sqcup f) = t_{\eex}(\zeta) + 2$.
  As such $t(Y) = t_{\eex}(\xi) > 0$ so $Y$ is not a $G$-core.

  \underline{n $\rightarrow$ B}: If $Y$ is not a $B$-core then there exists 
  an excisable border strip. As such we can find $p \in Y$ 
  such that the border of $N(p)$ is regularly colored, and in particular 
  the length of the border is equal to $n$. Note the length of the border is 
  also equal to $h(p)$ as it is the border of a net. Thus $Y$ is not an 
  $n$-core.

  \underline{G $\leftrightarrow$ MS}: By Theorem 4.9 of 
  \cite{Maclagan_Smith_2010} we have $\# T(M)$ is equal to the 
  dimension of $\eex^{[\rho]}$. By its definition this implies 
  $t(Y) = \# T(M)$.
\end{proof}
\end{prop}
In general these implications are sharp. To see this fix the $(5;1,3)$
action, choose $\rho \simeq \RR$, and consider the colored diagrams 
\begin{center}
  $Y_1 = $
  \begin{ytableau}
    *(green)\\
    *(red) & *(blue)\\
    *(white) & *(green)\\
  \end{ytableau}
  ~
  $Y_2 = $
  \begin{ytableau}
    \none\\
    \none\\
    *(white) & *(green) & *(yellow) & *(red) & *(blue)\\
  \end{ytableau}
\end{center}
Let $\xi_i$ denote the subscheme corresponding to $Y_i$. 
It can be verified $t(Y_2) = 2$ by deforming the 
subscheme away from the origin so in turn $t(Y_1) = 2$. Hence $Y_1$ is not a 
$G$-core. However it is a $B$-core as the border length is 4. For another 
implication, use the same coloring and consider the diagram
\begin{center}
  \begin{ytableau}
    *(red)\\
    *(white) & *(green) & *(yellow) & *(red)\\
  \end{ytableau}
\end{center}
Clearly it is equal to its own border and the hook length of the white square 
is 5. As it is not regularly colored this proves $B \neq n$.
However more can be said for the special case of $A_{n-1}$ actions.
\begin{prop}\label{CoreEq_nCore}
  $B \rightarrow n$ for the $A_{n-1}$ action.
  \begin{proof}
    Every length $n$ border strip for the $A_{n - 1}$ action is excisable
    by Propositions \ref{ExcisableSubsets} and \ref{RegColoring}. As 
    such the claim follows directly from the discussion in examples c and d 
    on pages 12-13 of \cite{Macdonald_2023}.
  \end{proof}
\end{prop}
Now the diagram of implications takes the form:
\[\begin{tikzcd}
	n & B & G & MS
	\arrow[tail reversed, from=1-1, to=1-2]
	\arrow[from=1-3, to=1-2]
	\arrow[tail reversed, from=1-3, to=1-4]
\end{tikzcd}\]
The main technical result for this section is:
\begin{theorem}\label{AllCoreEq}
  $n \rightarrow MS$ for the $A_{n-1}$ action.
\end{theorem}
Once this is proven, the above set of implications has a 
closed cycle and hence all four definitions of cores are logically 
equivalent! The proof of this implication and consequences thereof will 
constitute the remainder of this paper. Before doing so, some additional 
remarks are in order. 

Fix a choice of monomial ideal $M$ and let $\{x^{p_i} y^{q_i}\}_{i = 0}^m$ 
be a minimal set of generators with the same assumptions as above. Some 
additional shorthand which will be employed is: (i) $(i;u,v) \geq 0$ denotes 
a nonnegative arrow, (ii) $T_i(M) \subset T_{\geq 0}(M)$ consists of all of 
nonnegative significant arrows based at $(p_i, q_i)$. 
Next we will say $S \subset Y$ is a \textbf{shaft} of $Y$ if 
$(p_i,q_i-1), (u,v) \in S$ and $S = \{(a_i, b_i)\}_{i = 1}^k$ where $S$ is 
determined by a staircase sequence $C = (c_j)_{j = 0}^k$ for some $k$, 
called denoted $\ell(S)$ and called its \emph{length}.

\subsection{Technical Lemmas}
\begin{lemma}\label{BorderPush}
    Suppose $(i;u,v) \in T_i(M)$ and $(u+1, v+1) \in Y$. Then 
    $(i; u+1, v+1) \in T_i(M)$.
    \begin{proof}
        Note $\CC_n(u+1, v+1) = \CC_n(u,v) = \CC_{n}(p_i, q_i)$
        and hence $(i; u+1, v+1)$ is an arrow of $M$. Next, 
        since $u < u + 1$ and $v < v + 1$ it follows 
        $0 \leq (i; u, v) \leq (i; u+1, v+1)$. Now assume 
        $(i;u,v) \in T_{\geq 0}(M)$. In particular $i > 0$ and 
        $(u + p_{i - 1} - p_i, v) \notin Y$. As $Y$ is a diagram clearly 
        $(u + 1 + p_{i - 1} - p_i, v + 1) \notin Y$.
        Hence $(i; u+ 1, v + 1) \in T_{i}(M)$ as claimed. 
    \end{proof}
\end{lemma}
On induction the lemma implies for $k$ maximal with respect to 
$(i; u + k, v + k) \in T_i(M)$ that $(u + k, v + k)$ is contained 
in the border of $Y$. To proceed we may define flips on shafts as 
was done for staircase sequences as both are given by binary sequences. 
Furthermore, a negative flip will satisfy $\ff_j^{-}(S) = S$ if 
$c_j = 0$ and $\ff_j^{-}(S) = (c_i^*)$ where 
$c_i^* = c_i$ if $i \neq j$ and $c_j^* = 0$.
As was shown before the color type of the head of the arrow 
is preserved since necessarily the head is equal to $(u,v)$ or 
$(u \pm 1, v \pm 1)$. As such $\ff_j^{\pm}(S)$ is the shaft of some arrow 
of $Y$ so long as the head is contained in $Y$. 
\begin{lemma}\label{ShaftLength}
    Let $w = (i; u,v)$ be a nonnegative arrow of $Y$. If $S$ is a shaft of 
    $w$ then $n \mid \ell(S)$.
    \begin{proof}
      Similar to the proof of Lemma \ref{RegColoring} we can apply a 
      seequence of positive and negative flips to $S$ to obtain an arrow 
      either of the form $(i; 0,v)$ or $(i;u,0)$ and shafts $S_q, S_p$ 
      defined below. In the first case there is a unique shaft of 
      the arrow $S_q = \{(p_i, q_i - i)\}$ and necessarily $v < q_i$. Since 
      $\CC_n(p_i, q_i) = \CC_n(u,v) = \CC_n(p_i, v)$ this implies 
      $q_i = v \modn$ so $q_i = v + kn$ for some $k \in \N$; By comparison it 
      is easy to see $\ell(S) = kn$. In the second case define 
      \begin{equation*}
          S_p = \{(0,q_i - j) \mid 1 \leq i \leq q_i\} \cup \{(j,0) \mid 1 \leq j \leq n\}
      \end{equation*}
      Along the first column the shaft passes through $(0,v)$ and thus it 
      suffices to show $n \mid u + v$ since this is equal to the number of 
      squares between both points. As 
      $\CC_n(u,0) = \CC_n(p_i,q_i) = \CC_n(0,v)$ necessarily 
      $u + 0 = 0 - v \modn$. Upon rearranging 
      $u + v = 0 \modn$, as $u, v \geq 0$ this implies 
      $u + v = k'n$. Thus $\ell(S_p) = kn + k'n$ so $n \mid \ell(S_p)$. 
      Finally as every shaft of $(i;u,v)$ has a length equal to either 
      $\ell(S_q)$ or $\ell(S_p)$ this implies $n \mid \ell(S)$ for all 
      $S$.      
    \end{proof}
\end{lemma}

\subsection{The Proof}

Now we will prove the theorem. Fix a choice of diagram $Y$ colored with 
respect to $\CC_n$. Note $Y$ is an $n$-core iff $Y^t$ is an $n$-core. 
As this swaps the roles of nonnegative and nonpositive arrows it suffices to 
show $T_{\geq 0}(M) = \emptyset$ for all $Y$. The overarching strategy for 
the remaining argument is to employ an inductive argument on the number of 
minimal generators of $M$.

\underline{\textbf{The Base Case}}:
Assume $T_{\geq 0}(M) \neq \emptyset$ and $m = 1$ so $M = (x^a, y^b)$. 
Pick a non-negative significant arrow $w$. As there are two generators 
of $M$ it follows from the definitions that $w = (1; u,v)$ for some 
$(u,v) \in Y$. By Lemma \ref{BorderPush} without loss of generality 
$(u,v) \in B(Y)$. Note $h(0,v)$ is equal to the length of the border strip 
starting at $(0,b-1)$ and terminating at $(a-1,v)$ since the diagram is a 
rectangle. Next, I claim $a + b < n$. Otherwise the length of the border is 
at least $n$, from which a regularly colored border strip is easily found. 

As the diagram is a rectangle every border strip is the border of some 
net, and hence $Y$ contains an excisable border strip. As $n=B$ it follows
$Y$ is not an $n$-core which contradicts the hypothesis. Next as 
$h(0,0) = a + b$ and every other hook is bounded by $h(0,0)$ we can see 
$h(0,v) < n$ for all $p \in Y$. Finally, set 
$$S = \{(0,b - i) \mid 1 \leq i \leq b - v\} \cup \{(j, v) \mid 1 \leq j \leq a - 1\}$$ 
Clearly this is a shaft of $(i; u, v)$ and by comparison it follows
$\ell(S) = h(0,v) \leq n$. In fact equality holds since $n \mid \ell(S)$
by Lemma \ref{ShaftLength}. Finally as $h(u,v) = \ell(S)$ this 
implies $Y$ cannot be an $n$-core so the base case is proven.

\underline{\textbf{The Inductive Step}}:
Assume the theorem is true for all diagrams whose monomial ideal
has less than $m$ elements in a minimal generating set. Let $Y$ be the 
diagram corresponding to $M = (\{x^{p_i} y^{q_i}\}_{i = 0}^m)$ and 
set $Y'$ to be the translation of $N(0,p_{m-1})$ by the vector 
$(0,-p_{m - 1})$. From the definition of nets necessarily $Y'$ is a diagram.
Let $M'$ be the associated monomial ideal and note its minimal generators are
\begin{equation*}
    M' = (x^{p_0 - p_{m - 1}}, x^{p_1 - p_{m - 1}} y^{q_1}, ..., x^{p_{m - 2} - p_{m - 1}} y^{q_{m - 2}}, y^{q_{m - 1}})
\end{equation*}
The hook lengths of $Y'$ are equal to the hook lengths of $Y$ for
the corresponding elements of the net. As such $Y'$ is an $n$-core since
$Y$ is an $n$-core by assumption. By the inductive hypothesis we have
$T_{\geq 0}(Y') = \emptyset$ as $m - 1 < m$.
\begin{prop}
    For $0 < i < m$, if $Y$ is an $n$-core then $T_i(M) = \emptyset$. 
    \begin{proof}
        The assumption $0 < i$ is not restrictive since by definition a 
        nonnegative significant arrow must satisfy $0 < i$.
        Assume there exists $(i; u,v) \in T_{\geq 0}(M)$ with $0 < i < m$ so 
        there is an arrow between $(p_i, q_i)$ and $(u,v)$. As 
        $(p_i - p_{m - 1}, q_{i})$ is a minimal square not contained in $Y$ 
        and $\CC_n(p_i, q_i) = \CC_{n}(u,v)$, necessarily 
        $\CC_n(p_i - p_{m - 1}, q_i) = \CC_n(u - p_{m - 1}, v)$ so 
        $(i; u - p_{m - 1}, v)$ is an arrow of $Y'$.   
        Furthermore the arrow is nonnegative as $u \geq p_i$ implies 
        $u - p_{m - 1} \geq p_i - p_{m - 1}$. It can be verified
        $x^{u + p_{i - 1} - p_i} y^v \in M$ implies 
        $x^{u - p_{m - 1} + p_{i - 1} - p_i} y^v \in M'$. As such
        $(i; u - p_{m -1}, v) \in T_{\geq 0}(M')$ 
        which contradicts $T_{\geq 0}(M')$. Hence 
        $(i;u,v)$ cannot exist so $T_i(M) = \emptyset$.
    \end{proof}
\end{prop}
From the Lemma and the decomposition $T_{\geq 0}(M) = \bigcup_i T_i(M)$
it follows $T_{\geq 0}(M) = T_m(M)$. Thus to prove the theorem
it suffices to show $T_m(M) = \emptyset$. 
\begin{prop}
    If $Y$ is an $n$-core then $T_m(M) = \emptyset$.
    \begin{proof}
        Note if $w \in T_m(M)$ then necessarily $w = (m; u,v)$ for some 
        $(u,v) \in Y$. By Lemma \ref{BorderPush} assume 
        $(u,v) \in B(Y)$ and define a shaft of $(m;u,v)$ by 
        \begin{equation*}
          S = \{(i, q_{m} - 1) \mid 0 \leq i \leq u\} \cup \{(u, q_{m} - i) \mid 1 \leq i \leq q_m - v\}
        \end{equation*}
        Let $k = \min\{j \in \N \mid x^{p_k} y^v \in M\}$. 
        If $p_k = u + 1$ then $S$ is a hook of $(0,v)$ so 
        $n \mid \ell(S)$ by Lemma \ref{ShaftLength} which implies $(0,v)$ is 
        an $n$-hook. This contradicts the assumption that $Y$ is an $n$-core
        so necessarily $p_{k} > u + 1$. As the arrow is significant and 
        nonnegative, it is the case $x^{u + p_{m-1}} y^v \in M$ so 
        $p_k \leq u + p_{m - 1}$. Hence $p_k - u \leq p_{m - 1}$. 
        
        Next let $S'$ be the translation of $S$
        by the vector $(p_k - u,0)$. If $(i, 0) \in S$ then 
        $i \leq q_{m}, p_k - u \leq p_{m - 1}$ so $(i,p_k - u) \in Y$ for all 
        such $i$. Similarly $(i + p_k - u, v) \in Y$ for all $i$ in the 
        horizontal leg. In total $S + (p_k - u, 0) \subset Y$. 
        Moreso this translates the shaft so the
        uppermost and rightmost legs are contained in the border. Hence 
        $S'$ is the hook of $(p_k - u, v)$. Translating preserves the length
        so $\# S' = \ell(S)$. By Lemma \ref{ShaftLength} in turn 
        $S'$ is an $n$-hook which contradicts the assumption $Y$ is an 
        $n$-core. Thus $(m;u,v)$ cannot exist so $T_{m}(M) = \emptyset$.
    \end{proof}
\end{prop}
Both propositions together imply $T_{\geq 0}(M) = \emptyset$. By applying the 
same argument to $Y^t$ it follows $T_{\leq 0}(M) = \emptyset$ that 
$Y$ is an MS-core. Thus we have completed the proof of Theorem \ref{AllCoreEq} 
$\square$

\section{Applications of the Equivalences}

The final section will consist of consequences of Theorem \ref{AllCoreEq}.
Recall the $\CC_n: \N^2 \to \mu_n$ coloring is defined by 
$\CC_n(i,j) = i - j \modn$.

\subsection{Combinatorial Applications}

We can now derive some combinatorial applications, the first of which is 
another reinterpretation of the dimension formula. If the color type of a 
diagram $Y$ with respect to $\CC_n$ is $(a_0,...,a_{n-1})$ then we will
define the function
$$t_n(Y) = 2a_0 - \sum_{i = 0}^{n-1}(a_i - a_{i + 1})^2$$
where $a_n = a_0$; If $\zz$ is the substack corresponding to $Y$ then 
clearly $t_n(Y)$ is the dimension of the corresponding Gilbert scheme 
which contains $\zz$.
\begin{prop}\label{NumBorderStripFormula}
  Let $Y \subset \N^2$ be a Young diagram colored by $\CC_n$. Then $t_n(Y)$
  is equal to twice the number of border strips which can be removed before 
  obtaining an $n$-core. 
  \begin{proof}
    If $Y$ is an $n$-core then by the above equivalences it is also a 
    $G$-core. As such $t_n(Y) = 0$ from which the claim is vacuously true. 
    Assume without loss of generality $Y$ is not a $G$-core so it is not a 
    $B$-core and hence there exists an excisable border strip $B \subset Y$. 
    Let $(a_0,...,a_{n-1})$ be the color type of $Y$. The color type of 
    $Y-B$ is $(a_0 - 1, ..., a_{n - 1} - 1)$ as $B$ is regularly colored. 
    By the substituting the color type of $Y - B$ into the formula for 
    $t_n$ and simplifying the resulting expression it follows 
    $t_n(Y - B) = t_n(Y) - 2$. Set $B_1 = B$. 
    
    If $t_n(Y - B_1) > 0$ the same argument may 
    be repeated to construct $B_2 \subset Y - B_1$. Repeat these steps
    until $t_n(Y - \bigcup_i^k B_i) = 0$. As $G = B$ upon reaching this 
    step no more excisable border strips remain. On induction it can be 
    calculated $t_n(Y - \bigcup_i^k B_i) = t_n(Y) - 2 k$ for all $k$. 
    Thus taking the largest possible index and setting the equations 
    equal to each other yields $t_n(Y) = 2k$.
  \end{proof}
\end{prop}
For the next application, let $R_n \subset \N^n$ denote the subset consisting 
of $(a_0,...,a_{n-1}) \in \N^n$ such that there exists a diagram 
whose color type is $(a_0,...,a_{n-1})$.
\begin{prop}
There exists a bijection
\begin{equation*}
  \{n\text{-core partitions}\} \leftrightarrow \{(a_0,...,a_{n - 1}) \in R_n \mid 2a_0 = \sum_i (a_i - a_i)^2\}
\end{equation*}
\begin{proof}
      Let $p_1 \geq ... \geq p_m$ be an $n$-core partition and $Y$ be the 
      associated Young diagram with color type $p = (a_0, ..., a_{n - 1})$.  
      Define a map to $\N^n$ by sending $Y$ to $p$. This is well defined as 
      each partition gives rise to a unique diagram. 
      As $n=G$ it follows $t_n(Y) = 0$ so that in particular
      \begin{equation*}
        2 a_0 = \sum_i (a_i - a_{i + 1})^2
      \end{equation*}
      As such the image of the above map is contained in $R_n$, and from its 
      definition the map is necessarily surjective. 
      To show injectivity we need to show $n$-core diagrams are uniquely 
      determined by their color types. 
      
      To see this, let 
      $\rho \simeq \sum_i a_i \chi(i)$. As $t_n(Y) = 0$ it follows 
      $\eex^{[\rho]}$ is zero dimensional. Since Gilbert schemes are 
      quasiprojective in turn $\eex^{[\rho]}$ is a finite collection of 
      points. By Proposition 3.12 of \cite{Maclagan_Smith_2010} they are 
      also connected and hence $\eex^{[\rho]}$ is a singleton. In particular 
      if $\xi$ is the monomial subscheme corresponding to $Y$ then 
      $\eex^{[\rho]} = \{\xi\}$ from which it is easily seen the map is 
      injective.
\end{proof}
\end{prop}
The bijection above also gives an equational characterization of $n$-cores
in terms of solutions contained in $R_n$. It is then natural to ask if the 
bijection can be extended to \textbf{all} solutions of the given equation; 
The difficulty is ensuring if $(a_0, ..., a_{n - 1})$ is a solution then it 
is contained in $R_n$. For instance, $(0,1)$ cannot be the color type of any 
diagram colored by $\CC_2$ so trivially it follows $R_n \neq \N^2$. However 
we can prove:
\begin{prop}
  2-cores are in a bijective correspondence with all solutions of 
  $2a_0 = 2(a_0 - a_1)^2$.
  \begin{proof}
    We will omit the calculation that solutions of the equation are given 
    by $(j^2, j^2 \pm j)$ for $j \in \N$. Now let $b \in \{0,1\}$ and 
    define $Y_{j,j - b}$ to be the diagram corresponding to the partition 
    $(2j - b) + ... + 1$. In all cases $j$ counts the number of antidiagonals 
    of white squares. It can be verified the $i$-th antidiagonal has $2i - 1$ 
    squares so in total there are $j^2$ white squares. If $b = 0$ then there 
    are $j$ antidiagonals of black squares, of which the $k$-th one has 
    $2k$ squares. Thus there are $2 + ... + 2j = j^2 + j$ black squares 
    in $Y_{j,j}$ whence its color type is $(j^2, j^2 + j)$. The case 
    $b = 1$ is virtually identical and will be omitted. In particular 
    $(j^2, j^2 \pm j) \in R_2$ for all $j \in \N$, which is what we wanted 
    to prove.
  \end{proof}
\end{prop}
As the task of extending the bijection beyond $n = 2$ will not be undertaken 
in this note we will conclude this section with a stronger conjecture to 
return to in the sequel:
\begin{conjecture}
  $(a_0,...,a_{n - 1}) \in R_n$ iff $2a_0 \geq \sum_{i} (a_i - a_{i + 1})^2$
  for all $n \in \N$.
\end{conjecture}
One direction is clear since if $(a_0,...,a_{n - 1})$ comes from a 
diagram then $t_n(Y)$ is the dimension of the tangent 
space of some Gilbert scheme and thus is nonnegative. For the converse 
direction the key point would be to prove that if $\rho \in \Rep(G)$ 
has a corresponding color type $(a_0,...,a_{n-1})$ which satisfies the 
above inequality then $\eex^{[\rho]} \neq \emptyset$.

\subsection{Applications to Gilbert Schemes}\label{TopGilbSchemes}

With equality of the cores established we can also give an alternative 
interpretation to the topology of $\DD_n^{[\rho]}$ via colored 
diagrams. To begin, an interpretation of core diagrams will be given.

\medskip 
For a given representation $\rho$, use $\rho_{\text{min}}$ to denote the 
minimal subrepresentation of $\rho$ such that 
$\rho_{\text{min}} = \rho - M\RR$ for $M = \dim(\DD_n^{[\rho]})$; 
$\rho_{\text{min}}$ will be called the minimal core of $\rho$. The 
analysis when $M = 0$ corresponds to core diagrams. More precisely 
we have:
\begin{lemma}
  M = 0 iff $\DD_n^{[\rho]}$ is a singleton. Furthermore, 
  its unique point is a monomial subscheme $\xi$ such that $Y(\xi)$ 
  is a core diagram.
  \begin{proof}
    The first claim is immediate since $M = \dim(\DD_n^{[\rho]})$ and 
    $\DD_n^{[\rho]}$ is connected by Theorem 1.1 of 
    \cite{Maclagan_Smith_2010}. For the second claim, choose 
    $\xi \in \DD_n^{[\rho]}$. As Grobner deformations induce $G$-flat 
    families it follows $\text{in}(\xi) \in \DD_n^{[\rho]}$ for any 
    admissible monomial order. However Grobner deformations are 
    parameterized by $\aff^1$ so as $\DD_n^{[\rho]}$ is a singleton we 
    have $\xi = \text{in}(\xi)$, i.e., $\xi$ is a monomial subscheme.
    Finally $Y = Y(\xi)$ is a core diagram as 
    $t_n(Y) = \dim(\eex^{[\rho]}) = 0$.
  \end{proof}
\end{lemma}
Now suppose $M > 0$, let $\xi \in \DD_n^{[\rho]}$ be a monomial subscheme, 
and set $Y = Y(\xi)$. As $Y$ is not a $G$-core there exists an excisable 
border strip. By Proposition \ref{ExcisableSubsets} a stable disentanglement  
may be constructed. Let $\zeta_1$ be the base of the disentanglement
and $Y_1 = Y(\zeta_1)$. Arguing as in the proof of Corollary 
\ref{NumBorderStripFormula} in fact for each $k \leq M$ there exists a 
stable disentanglement of gauge $k$ and associated base $\zeta_k$ with 
$Y_k = Y(\zeta_k)$. Reinterpreting the above set of steps, by a standard 
functor of points argument there exists a morphism 
\begin{equation*}
    \mathcal{U}_k: \GG_{\rho - k \RR} \times \Delta^{[k \RR]} \rightarrow \DD_n^{[\rho]}
\end{equation*}
which on closed points is given by 
$\mathcal{U}_k(\zeta, f) = \zeta \sqcup f$. 
Set $O_k = \text{im}(\mathcal{U}_k)$ and $C_k = \overline{O_k}$. 
Each $\GG_{\rho - k \RR}$ is connected by Propositions 3.10 and 3.12 of 
\cite{Maclagan_Smith_2010}, and $\Delta^{[k\RR]} \simeq \delta^{[k]}$ is 
connected by Proposition 2.3 of \cite{Fogarty_1968} so both $O_k$ and 
$C_k$ are connected for each $k$. Moreso as everything involved is 
quasiprojective over $\C$, $O_k$ is a constructible subset by Chevalleys 
theorem and we will endow it with the induced reduced subscheme structure.
As each substack of $\DD_n$ may be uniquely written as 
$\zz = \zz_0 \sqcup (\zz \cap \Delta)$
where $\text{Supp}(\zz_0) = 0$, it follows for all $\rho$ there is a 
decomposition 
\begin{equation*}
  \DD_n^{[\rho]} = \bigsqcup_{k = 0}^{M} O_k
\end{equation*}
Paired with the previous paragraph, for each $k$ we may decorate it with 
the set of diagrams of size $\deg(\rho - k \RR)$ and colored with 
$\CC_n$; Passing from $O_k$ to $C_k$ amounts to adding the 
border strips back to each $Y_k$. In total we have proven:
\begin{prop}
  Let $n \in \N$, $\rho \in \Rep(G)$, and 
  $\rho_{\operatorname{min}} = \rho - M \RR$ be the minimal core of $\rho$. 
  There is a partition $\DD_n^{[\rho]} = \bigsqcup_{k = 0}^M O_k$ such that 
  $O_M = \GG_{\rho_{\operatorname{min}}} \times \Delta^{[M\RR]}$. Furthermore
  the elements of the partition are connected smooth varieties and are in a 
  bijective correspondence with diagrams of size $\deg(\rho - k\RR)$ colored by 
  $\CC_n$.
\end{prop}
In fact we can regard the complete collection of diagrams of a given size 
as being an equivalence class of diagrams up to algebraic deformation, i.e., 
a connected component of $\GG_{\rho - k \RR}$.


\printbibliography

\bigskip
University of Maryland
Department of Mathematics.
4176 Campus Drive,
William E. Kirwan Hall, College Park, Maryland 20742

\textit{Email Addresses}: \emph{sbachhub@umd.edu, amingh@umd.edu}

\end{document}